\documentclass{amsproc}%{aims}
\usepackage{amsmath}
  \usepackage{paralist}
  \usepackage{graphics} %% add this and next lines if pictures should be in esp format
  \usepackage{epsfig} %For pictures: screened artwork should be set up with an 85 or 100 line screen
\usepackage{graphicx}  \usepackage{epstopdf}%This is to transfer .eps figure to .pdf figure; please compile your paper using PDFLeTex or PDFTeXify. 
 \usepackage[colorlinks=true]{hyperref}
\hypersetup{urlcolor=blue, citecolor=red}
\usepackage{amsfonts,amssymb,stmaryrd,amsthm}
\usepackage{subfigure}
\usepackage{times} 
\usepackage{amscd}
\usepackage{yfonts,mathrsfs}
\usepackage[cmtip,arrow]{xy}
\usepackage{pb-diagram,pb-xy}
\usepackage{titletoc}
\usepackage{palatino}
\usepackage{bm}
\usepackage{color}

\newtheorem{theorem}{Theorem}[section]
\newtheorem{corollary}[theorem]{Corollary}

\newtheorem{lemma}[theorem]{Lemma}
\newtheorem{proposition}[theorem]{Proposition}

\theoremstyle{definition}
\newtheorem{definition}[theorem]{Definition}
\newtheorem{remark}[theorem]{Remark}

\title[Lattice Structures for Attractors  II]{Lattice Structures for Attractors  II\footnote{\today}}

\author[W.D. Kalies, K. Mischaikow, and R.C.A.M. Vandervorst]{}


 \email{mischaik@math.rutgers.edu}
 \email{wkalies@fau.edu}
 \email{vdvorst@few.vu.nl}

\thanks{The first author is partially supported by NSF grant NFS-DMS-0914995, the second author is 
partially supported by NSF grants NSF-DMS-0835621, 0915019, 1125174, 1248071, and contracts from AFOSR and DARPA.
The present work is part of the third authors activities within CAST, a Research Network Program of the European Science Foundation ESF}

\newcommand{\rc}{{\scriptscriptstyle\#}}
\newenvironment{fig}
{\begin{figure}[hbt]}%\stepcounter{theorem}}
{\end{figure}}
\newcommand{\bfig}{\begin{fig}}
\newcommand{\efig}{\end{fig}}

\newcommand{\setof}[1]{\left\{ {#1}\right\}}
\newcommand{\mvmap}{\raisebox{-0.2ex}{$\,\overrightarrow{\to}\,$}}

\newcommand{\N}{{\mathbb{N}}}
\newcommand{\R}{{\mathbb{R}}}
\newcommand{\T}{{\mathbb{T}}}
\newcommand{\Z}{{\mathbb{Z}}}

\newcommand{\cA}{{\mathcal A}}

\newcommand{\cF}{{\mathcal F}}
\newcommand{\cG}{{\mathcal X}} % CHANGE THIS!!!!

\newcommand{\cN}{{\mathcal N}}

\newcommand{\cR}{{\mathcal R}}
\newcommand{\cS}{{\mathcal S}}

\newcommand{\cU}{{\mathcal U}}
\newcommand{\cV}{{\mathcal V}}
\newcommand{\cW}{{\mathcal W}}
\newcommand{\cX}{{\mathcal X}}

\newcommand{\sA}{{\mathsf A}}
\newcommand{\sB}{{\mathsf B}}

\newcommand{\sH}{{\mathsf H}}

\newcommand{\sJ}{{\mathsf J}}
\newcommand{\sK}{{\mathsf K}}
\newcommand{\sL}{{\mathsf L}}

\newcommand{\sO}{{\mathsf O}}
\newcommand{\sP}{{\mathsf P}}

\newcommand{\sR}{{\mathsf R}}

\newcommand{\scrR}{{\mathscr R}}

\newcommand{\cFm}{{\mathcal F}_{{\rm o}}}
\newcommand{\cFmn}{{\mathcal F}_{{\rm o},n}}
\newcommand{\cFme}{{\mathcal F}_{{\rho}}}
\newcommand{\cFmen}{{\mathcal F}_{{\rho}_n}}

\newcommand{\sAtt}{{\mathsf{ Att}}}
\newcommand{\sRep}{{\mathsf{ Rep}}}

\newcommand{\sInvset}{{\mathsf{ Invset}}}

\newcommand{\IS}{{\mathsf{ Invset}}}

\newcommand{\sGrid}{{\mathsf{Grid}}}

\newcommand{\sANbhd}{{\mathsf{ ANbhd}}}
\newcommand{\sRNbhd}{{\mathsf{ RNbhd}}}
\newcommand{\sANbhdR}{{\mathsf{ ANbhd}}_{\scrR}}
\newcommand{\sRNbhdR}{{\mathsf{ RNbhd}}_{\scrR}}

\newcommand{\subFR}{{{\rm sub}_F\mathscr{R}}}

\newcommand{\sSet}{{\mathsf{Set}}}
\newcommand{\sASet}{{\mathsf{ASet}}}
\newcommand{\sRSet}{{\mathsf{RSet}}}

\newcommand{\bomega}{{\bm{\omega}}}
\newcommand{\balpha}{{\bm{\alpha}}}

\newcommand{\Chi}{\raise .75ex\hbox{$\chi$}}

\newcommand{\pred}[1]{\overleftarrow{#1}}

\newcommand{\down}{\downarrow\!}

\newcommand{\vgln}{\lb\begin{array}{rcl}}
\newcommand{\eindvgln}{\end{array}\right.}

\newcommand{\dist}{{\rm dist}\,}

\newcommand{\cl}{{\rm cl}\,}%{\text{cl}}

\newcommand{\diam}{{\rm diam}\,}
\newcommand{\Int}{\mbox{\rm int\,}}
\newcommand{\Inv}{\mbox{\rm Inv\,}}

\newcommand{\supp}[1]{\left| {#1}\right|}

\newcommand{\id}{\mathop{\rm id }\nolimits}

\newcommand{\cov}{\mathop{\rm cov}}

\begin{document}
\begin{sloppypar}
\maketitle

% Enter the first author's name and address:
\centerline{\scshape William D. Kalies }
\medskip
{\footnotesize
% please put the address of the first author
%\centerline{Department of Mathematical Sciences]
 \centerline{Florida Atlantic University}
   \centerline{777 Glades Road}
   \centerline{Boca Raton, FL 33431, USA}
} % Do not forget to end the {\footnotesize by the sign }

\medskip

\centerline{\scshape Konstantin Mischaikow}
\medskip
{\footnotesize
 % please put the address of the second  and third author
% \centerline{Department of Mathematics/BioMaPS Institute}
 \centerline{Rutgers University}
   \centerline{110 Frelinghusen Road}
   \centerline{Piscataway, NJ 08854, USA}
}

\medskip

\centerline{\scshape Robert C.A.M. Vandervorst}
\medskip
{\footnotesize
 % please put the address of the second  and third author
% \centerline{Department of Mathematics}
 \centerline{ VU University}
   \centerline{De Boelelaan 1081a}
   \centerline{1081 HV, Amsterdam, The Netherlands}
}

\bigskip

% The name of the associate editor will be entered by an editorial staff
% "Communicated by the associate editor name" is not needed for special issue.
% \centerline{(Communicated by the associate editor name)}

\begin{abstract}
The algebraic structure of the attractors in a dynamical system determine much of its global dynamics.
The collection of all attractors has a natural lattice structure, and this 
structure can be detected through attracting neighborhoods,
which can in principle be computed. Indeed, there has been much recent work
on developing and implementing general computational algorithms for global dynamics, which
are capable of computing attracting neighborhoods efficiently. Here we address the question 
of whether all of the algebraic structure of attractors can be captured by these methods.
%
%
%
%We discuss  basic lattice  structures of attractors and repellers in combinatorial multivalued dynamical systems.
%The structure of distributive lattices allows for an algebraic treatment of gradient-like dynamics 
%in combinatorial multivalued dynamical systems. Combinatorial multivalued dynamical systems are used as discretizations of 
%general dynamical systems. Given a resolution all dynamical characteristics in attractors and repellers are retrieved in
%the appropriate combinatorial multivalued discretizations. The main results in this paper entail various convergence theorems
%concerning attractor lattices.
\end{abstract}

% !TEX root = ./attractors-IIs2.tex

\section{Introduction}
\label{sec:intro}
The issue of computability in the context of nonlinear dynamics has recently received considerable attention; see for example \cite{braverman,bournez} and references therein. An important implication of this work is that the topological structure of invariant sets need not be computable.
Perhaps this is not surprising, given that the work over the last century has clearly demonstrated the incredible diversity and complexity of invariant sets.
%\corrl I'm not sure what the following sentence is supposed to say -- what do arbitrary and necessary mean? -- BK 9/6/14 -- revised BK 9/16/14, ---- I agree, RC 9/17/14  I have revised the revised version. meaningful followed by not useful seems strange. KM 9/18/14
%<<
%One interpretation of these results for practical applications  is that analyzing dynamics by computing a variety of meaningful invariant sets 
%may lead to a level of computations that is too fine to be useful and perhaps ultimately unattainable.
%||
One interpretation of these results for practical applications  is that analyzing dynamics by computing invariant sets 
may lead to a level of computations that is too fine to be useful and perhaps ultimately unattainable.

With this in mind we consider the question of the computation of coarse dynamical structures for the following rather general setting.
A \emph{dynamical system} on a topological space $X$ is a continuous map $\varphi: \T^+ \times X \to X$ that
satisfies
\begin{enumerate}
\item [(i)] $\varphi(0,x) = x$ for all $x\in X$, and
\item [(ii)] $\varphi(t,\varphi(s,x)) = \varphi(t+s,x)$ for all $s,t\in \T^+$ and for all $x\in X$,
\end{enumerate}
where $\T$ denotes the time domain, which is  either $\Z$ or $\R$ and $\T^+:=\setof{t\in\T\mid t\geq 0}$.
As is discussed in detail in Section~\ref{subsec:approximationDS}, for the results presented in this paper there is no loss of generality in assuming the dynamics is generated by the continuous function $f\colon X\to X$ where $f(\cdot) := \varphi(1,\cdot)$.  
The most significant assumption we make is that $X$ is a compact metric space.  
We emphasize that we do {\em not} assume that $f$ is injective nor surjective. 

Recall that a set $U\subset X$ is an \emph{attracting neighborhood} for  $f$ if
\[
\omega(U,f):= \bigcap_{n\in\Z^+} \cl\left(\bigcup_{k=n}^\infty f^k(U)  \right) \subset \Int(U).
\]  
A set $A\subset X$ is an \emph{attractor}  if there exists an attracting neighborhood $U$  such that $A = \omega(U,f)$.  
The sets of all attracting neighborhoods and all attractors are denoted by $\sANbhd(X,f)$ and $\sAtt(X,f)$, respectively. 
We remark that in general a given system can have at most a countably infinite number of attractors.

Attractors are central to the study of nonlinear dynamics for at least two reasons.  
First, they are the invariant sets that arise from the asymptotic dynamics of regions of phase space, thus they capture the ``observable" dynamics.
Second, they are intimately related to the structure of the global dynamics.  More precisely, recall that Conley's fundamental decomposition theorem \cite{robinson} states that 
 the dynamics is gradient like outside of the chain recurrent set.
%for any invariant set the flow is gradient-like off its chain recurrent set.
Furthermore, the chain recurrent set can be characterized using the set of attractors and their dual repellers.
With this in mind in \cite{KMV0} we discuss a combinatorial approach for identifying attracting neighborhoods and demonstrate that, even though there may be infinitely many attractors, it is possible to obtain arbitrarily good approximations in phase space of these attractors.  This in turn provides a constructive method for obtaining arbitrarily good approximations of the chain recurrent set.

From the perspective of understanding the dynamics of nonlinear models one  encounters the issue of minimal scales. 
Every model has a scale below which the model is no longer valid.
Any given numerical simulation has a minimal scale, and there is a maximal resolution for experimental measurements.
Especially in the latter  case, the maximal relevant resolution is often dependent on the location in phase space.
This issue of scale motivates recent work \cite{database,bush2012combinatorial,Ban,lyapunov} that focuses on rigorously computing global dynamical structures with an a priori choice of  maximal resolution of measurement.

Given a fixed scale there can be at most a finite subset of attractors that are observable. 
This suggests that understanding finite resolution dynamics requires a deeper understanding of the structure of the set of all attractors.
In \cite{KMV-1a} we  prove that $\sAtt(X,f)$ and $\sANbhd(X,f)$  are bounded, distributive lattices. 
The lattice operations for $\sANbhd(X,f)$ are straightforward, $\vee = \cup$ and $\wedge = \cap$.
The operations for $\sAtt(X,f)$ are more subtle; $\vee = \cup$, but the $\wedge$ operation is given by $A_1\wedge A_2=\omega(A_1\cap A_2,f)$. 
A consequence of the lattice structure is that a finite set of attractors generates a finite sublattice $\sA\subset \sAtt(X,f)$ of attractors.

By definition $\omega(\cdot,f)\colon \sANbhd(X,\varphi)\to \sAtt(X,f)$ is surjection. In \cite{KMV-1a} we show that this is, in fact, a lattice epimorphism.  
However, this fact does not provide any information, in and of itself, as to the structural relationship between a given finite lattice of attractors $\sA$, the dynamic information of interest, and the set of attracting neighborhoods $\omega(\cdot,f)^{-1}(\sA) \subset \sANbhd(X,f)$, which consist of the potentially observable or computational objects.
This is resolved by the following theorem which proves that the lattice structure of invariant dynamics of interest, namely attractors, is contained within the lattice structure of the observable or computable   dynamics, namely attracting neighborhoods.

%\corrl I don't think definitions belong in theorems. Also there is no part (ii). KM 9/18/14
%<<
%\begin{theorem}
%\label{thm:kmv1}
%\cite[Theorem 1.2]{KMV-1a}
%Let $\imath$ denote the inclusion map.
%(i) For every finite  sublattice $\sA \subset \sAtt(X,f)$, there exists a  lattice monomorphism $k$ such that the following diagram 
%\[
%\begin{diagram}
%\node{~} \node{\sANbhd(X,f)} \arrow{s,r,A}{\omega(\cdot,f)}\\
%\node{\sA}  \arrow{e,l,V}{\imath} \arrow{ne,l,..,V}{k}  \node{\sAtt(X,f)} 
%\end{diagram}
%\]
%commutes. The homomorphism $k$ is called a lift of $\imath$ through $\omega(\cdot,f)$. 
%\end{theorem}
%||
\begin{theorem}
\label{thm:kmv1}
\cite[Theorem 1.2]{KMV-1a}
Let $\imath$ denote the inclusion map.
 For every finite  sublattice $\sA \subset \sAtt(X,f)$, there exists a  lattice monomorphism $k$ such that the following diagram 
\[
\begin{diagram}
\node{~} \node{\sANbhd(X,f)} \arrow{s,r,A}{\omega(\cdot,f)}\\
\node{\sA}  \arrow{e,l,V}{\imath} \arrow{ne,l,..,V}{k}  \node{\sAtt(X,f)} 
\end{diagram}
\]
commutes. 
\end{theorem}
The homomorphism $k$ is called a lift of $\imath$ through $\omega(\cdot,f)$. 
The proof of Theorem~\ref{thm:kmv1} is nontrivial, especially since we are not assuming that $f$ is injective nor surjective.  Thus  \cite{KMV-1a} contains a detailed discussion and development of the definitions and properties of many of the standard dynamical concepts such as attractors, repellers, invariant sets, etc.~as well as corresponding neighborhoods of these objects. We do not repeat them in this paper, but recall them as necessary.

As suggested at the beginning of this introduction, the focus of this paper is on computation.
The computational methods we are interested in analyzing are based on a finite discretization, indexed by $\cX$, of the phase space $X$ and the computation of an outer approximation of a map $f\colon X\to X$ by a combinatorial multivalued map $\cF\colon\cX\mvmap\cX$ (see Section~\ref{sec:cgt}).
Observe that a combinatorial multivalued map is equivalent to a finite directed graph. 
The latter interpretation is useful from the perspective of algorithms, but treating $\cF$ as a map provides intuition as to how to define important dynamical analogues in the discrete setting. This is discussed in detail in Section~\ref{sec:comb-syst}, but  we also point out the work in \cite{McG,McGW} on closed relations.

Given our focus on attractors there are three structures arising from combinatorial dynamics that are of particular interest:~forward invariant sets, $\sInvset^+(\cX,\cF):=\setof{\cS\subset \cX\mid \cF(\cS)\subset (\cS)}$;
 attracting sets, $\sASet(\cX,\cF):=\setof{\cU\subset\cX \mid \bomega(\cU,\cF) \subset \cU}$; and 
 attractors, $\sAtt(\cX,\cF) := \setof{\cA\subset\cX \mid \cF(\cA)=\cA}$.
In Section~\ref{sec:comb-syst} we assign appropriate lattice structures to these sets. Note that these lattices are explicitly computable, since they are defined in terms of elementary operations on a finite directed graph.

To easily pass between the combinatorial  and continuous dynamics, we insist that the discretization of phase space be done with regular closed sets.  
Given a compact metric space $X$, the family of all regular closed sets $\scrR(X)$ forms a Boolean algebra.
As this gives rise to technical issues, it is important to note that the lattice operations for $\scrR(X)$ differ from those of $\sSet(X)$, in particular $\vee=\cup$ and $\wedge=\cl(\Int(\cdot)\cap\Int(\cdot))$.
The atoms of any finite sublattice of $\scrR(X)$ form a {\em grid}, which provides an appropriate discretization of the phase space $X$.
As indicated above the grid is indexed by $\cX$. We pass from subsets of $\cX$ to subsets of $X$, by means of an evaluation map
$\supp{\cdot}\colon \cX \to \scrR(X)$.  

Observe that, as a consequence of the discretization procedure, our computations can only represent elements of $\scrR(X)$, which is a strict subset of $\sSet(X)$.  
We denote the family of attracting neighborhoods of $f$ that are regular closed sets of $X$ by $\sANbhdR(X,f)$.
Even though $\sANbhdR(X,f)\subset \sANbhd(X,f)$, this inclusion is not a lattice homomorphism since, the lattice operations are different.  
Furthermore, for a fixed multivalued map $\cF$, an outer approximation for $f$, the evaluation map $\supp{\cdot}$ maps $\sASet(\cX,\cF)$ to a strict subset of $\sANbhdR(X,f)$.  
The lattice homomorphisms that relate the above mentioned lattices produce the following commutative diagram (see Remark~\ref{rem:comm-lift11} for the analogue for $\varphi$)
\begin{equation}
\label{diag:comm-lift11}
\begin{diagram}
\node{\IS^+(\cX,\cF)}\arrow{se,r,A}{\bomega(\cdot,\cF)}\arrow{e,l,V}{\imath}\node{\sASet(\cX,\cF)}\arrow{e,l,V}{|\cdot|}\arrow{s,r,A}{\bomega(\cdot,\cF)}\node{\sANbhdR(X,f)}\arrow{e,l,A}{\omega(\cdot,f)}\node{\sAtt(X,f)}\\
\node{}\node{\sAtt(\cX,\cF)}\arrow{ene,r}{\omega(|\cdot|,f)}
\end{diagram}
\end{equation}
where $\imath$ is the inclusion map.

Returning to the question of computability, the analoguous result to Theorem~\ref{thm:kmv1} in the context of a multivalued outer approximation $\cF$ of $f$ is the existence of a lifting for either of the following commutative diagrams
\begin{equation}
\label{eq:twodiagrams}
\begin{diagram}
\node{~} \node{\sASet(\cX,\cF)} \arrow{s,r,A}{\omega(\cdot,f)}\\
\node{\sA}  \arrow{e,l,V}{\imath} \arrow{ne,l,..,V}{}  \node{\sAtt(X,f)} 
\end{diagram}
\qquad
\text{or}
\quad
\begin{diagram}
\node{~} \node{\IS^+(\cX,\cF)} \arrow{s,r,A}{\omega(\cdot,f)}\\
\node{\sA}  \arrow{e,l,V}{\imath} \arrow{ne,l,..,V}{}  \node{\sAtt(X,f)} 
\end{diagram}
\end{equation}
for a given finite sublattice of attractors $\sA$.  Observe that by \eqref{diag:comm-lift11} a lift for the second diagram implies a lift for the first.  

If $\sAtt(X,f)$ is an infinite lattice, then any given approximation $\cF\colon \cX\mvmap\cX$  captures 
only a finite sublattice $\sA\subset \sAtt(X,f)$.
With this in mind we consider sequences of multivalued maps $\cF_n\colon \cX_n\mvmap\cX_n$ that provide arbitrarily close approximations of $f$. 
There are essentially two types of sequences that we consider, one that involves a coherent refinement of the grids associated with $\cX_n$ and the other where it is only assumed that the diameter of the grid 
can be made arbitrarily small.
The first is most relevant if one considers a numerical scheme based on a systematic refinement of phase space.
The second is relevant if one wants to compare approximations performed using different types of discretizations of phase space.
With a coherent refinement scheme we are able to prove the existence of a lifting, Theorem~\ref{thm:conv-latt-1}, based on the second diagram of \eqref{eq:twodiagrams}. The more general result, Theorem~\ref{thm:conv-latt-3}, is based on the first diagram.

We conclude this introduction with a brief outline of the paper. 
Section~\ref{sec:comb-syst} introduces the dynamics of combinatorial multivalued maps and the appropriate lattice structures.
This includes the lattices of backward invariant sets $\sInvset^-(\cX,\cF)$, repelling sets $\sRSet(\cX,\cF)$, and repellers $\sRep(\cX,\cF)$.
The duality between these lattices, associated with backward dynamics, and those associated with forward dynamics is presented in the commutative diagram \eqref{eq:diag+-}.  
%The reader familiar with Conley's decomposition theory will not be surprised to see these structures included, nevertheless justification is provided below in
%the description of Section~\ref{sec:str-conv}.

In Section~\ref{sec:cont-graph} we focus on combinatorial multivalued maps as an approximation scheme for continuous nonlinear dynamics. 
We begin in Section~\ref{sec:cgt} by recasting the concept of grid \cite{Mrozek-grid} into the more general setting of regular closed subsets, %of a topological  space X, 
cf.\ \cite{Walker}.  
Section~\ref{subsec:at-rep-neigh} contains results, summarized for the most part by Theorems~\ref{thm:first-evalf} and \ref{thm:commutative-diagram}, that relate the lattice structures of combinatorial systems with those of continuous systems.
Section~\ref{subsec:approximationDS} deals with the  issue of the approximation of dynamical systems where the time variable $\T = \R$.  As mentioned earlier in the introduction, one approach is to set $f(\cdot) = \varphi(\tau, \cdot)$, where if $\T = \R$ then it is permissible to choose any fixed $\tau >0$.  
For this approach the concept of outer approximation is sufficient.
The weakness of this approach is that from the perspective of obtaining optimal approximations it may be desirable to choose different values of $\tau$ on different regions of phase space.
An alternative approach developed in \cite{BKM} involves combinatorializing  the flow  via a triangulation of space and the multivalued mapping is defined by considering the behavior of the associated vector field on the vertices of the triangulation.
This method fits into our framework but requires the notion of a weak outer approximation as is demonstrated via the commutative diagram \eqref{diag:AR2a}.

Section~\ref{sec:str-conv} brings together the ideas of Sections~\ref{sec:comb-syst} and \ref{sec:cont-graph} to demonstrate the general computability of the lattices of interest. 
We begin in Section~\ref{subsec:conv-fam} with a discussion concerning the convergence  of outer approximations from a more classical numerical analysis perspective, i.e.\ tracking of individual orbits.
Section~\ref{subsec:resol} discusses the identification individual attractors or repellers using a outer approximations.
Section~\ref{subsec:PLG} returns to the issue of convergent sequences of outer approximations but from a lattice theoretic perspective. 
We also discuss Birkhoff's representation theorem for finite distributive lattices and discuss its application in the context of lifts.
Finally in Section~\ref{subset:lift-grid} we prove the desired lifting theorems, Theorem~\ref{thm:conv-latt-1} and  \ref{thm:conv-latt-3}.
The reader will immediately note that the details of the proofs are carried out in the context of repeller structures, Theorems~\ref{thm:conv-latt-2} and \ref{thm:lift}, and then duality is used to obtain Theorem~\ref{thm:conv-latt-1} and  \ref{thm:conv-latt-3}, respectively.
The reason for this is related to the lattice structures of $\sAtt(X,f)$ and $\sRep(X,f)$. In particular, $\wedge = \cap$ for  $\sRep(X,f)$, but not for $\sAtt(X,f)$.
This lack of symmetry arises from the fact that we do not assume that $f$ is either injective or surjective, and thus attractors and repellers have fundamentally different properties. The duality between between attractor/attracting neighborhoods and repellers/repelling neighborhoods is expressed in following diagram, cf.\ \eqref{diag:AR2}
\begin{equation}\label{diag:AR2b}
\begin{diagram}
\node{\sANbhd(X,f)} \arrow{e,l,<>}{^c} \arrow{s,l,A}{\displaystyle \omega}\node{\sRNbhd(X,f)} \arrow{s,r,A}{\displaystyle \alpha}\\
\node{\sAtt(X,f)} \arrow{e,l,<>}{^*}   \node{\sRep(X,f)} 
\end{diagram}
\end{equation}
where $\sRep(X,f)$ and $\sRNbhd(X,f)$ are the lattices of repellers and repelling neighborhoods, respectively. 

\begin{remark}
We include a variety of different lattices in this paper.
In each case the $\vee$ operation is simply the union of sets, but there are five different  $\wedge$  operations.
For the benefit of the reader we include the following tables, the first for topological structures and the second for combinatorial structures, as a simple summary.  
\vskip 12pt
\begin{center}
\begin{tabular}{|l|c|}
\hline
Lattice & $U \wedge V$\\
\hline
$\sAtt(X,f)$ & $\omega(U\cap V)$\\
$\sRep(X,f)$ & $U\cap V$\\
$\sANbhd(X,f)$ & $U\cap V$\\
$\sRNbhd(X,f)$ & $U\cap V$\\
$\sANbhdR(X,f)$ & $\cl(\Int(U)\cap\Int(V))$\\
$\sRNbhdR(X,f)$ & $\cl(\Int(U)\cap\Int(V))$\\
$\IS^\pm(X,f)$ & $U\cap V$\\
\hline
\end{tabular}
\qquad
\begin{tabular}{|l|c|}
\hline
Lattice & $\cU \wedge \cV$\\
\hline
$\sAtt(\cX,\cF)$ & $\bomega(\cU\cap \cV)$\\
$\sRep(\cX,\cF)$ & $\balpha(\cU\cap \cV)$\\
$\sASet(\cX,\cF)$ & $\cU\cap \cV$\\
$\sRSet(\cX,\cF)$ & $\cU\cap \cV$\\
$\sInvset^+(\cX,\cF)$ & $\cU\cap \cV$\\
$\sInvset^-(\cX,\cF)$ & $\cU\cap \cV$\\
%$\IS^\pm(\cX,\cF)$ & $\cU\cap \cV$\\
\hline
\end{tabular}
\end{center}
\vskip 12pt

\end{remark}

\section{Combinatorial systems}
\label{sec:comb-syst}

In this section we discuss the dynamics of combinatorial multivalued maps. We begin with  basic properties, especially those related to
the asymptotic dynamics. We then discuss attractors, repellers and the combinatorial equivalences of their neighborhoods. Finally we discuss the concept of attractor-repeller pairs in this combinatorial setting. 

\subsection{Combinatorial multivalued maps}

Let $\cX$ be a finite set of vertices.
To emphasize the fact that we are interested in dynamics we denote mappings $\cF \colon \cX \to \sSet(\cX)$ by  $\cF\colon \cX \mvmap \cX$ and refer to them as \emph{combinatorial multivalued mappings} on $\cX$.
The inverse image of a element 
$\xi \in \cX$ is defined by
 \begin{equation}
\label{eqn:inverse-cms}
 \cF^{-1}(\xi) := \{ \eta\in \cX~|~ \xi\in \cF(\eta)\},
\end{equation}
which generates a combinatorial multivalued mapping denoted by $\cF^{-1}\colon\cX\mvmap \cX$.

%\corrl We never discuss this so I suggest removing KM
%<<
%In developing algorithms, it is often useful to realize $\cF$ as a directed graph on the vertices $\cX$
%with an edge from $\xi$ to $\eta$ if $\eta\in\cF(\xi)$.
%||
%>>

\begin{definition}
\label{defn:total}
A multivalued mapping is \emph{left-total} if $\cF(\xi)\neq\varnothing$ for all $\xi \in \cX$ and \emph{right-total}
if $\cF^{-1}(\xi)\neq\varnothing$ for  all $\xi \in \cX$. A multivalued mapping is \emph{total} if it is both left- and right-total.
\end{definition}

The {\em $\bomega$-limit set} and {\em $\balpha$-limit set}   capture the asymptotic dynamics
 of a set $\cU\subset \cX$ under  $\cF\colon \cX\mvmap \cX$ and  are defined by
\[
\bomega(\cU) = \bigcap_{k\ge 0} \bigcup_{n\geq k} \cF^n(\cU) 
\quad\hbox{and}\quad 
\balpha(\cU) = \bigcap_{k\le 0} \bigcup_{n\leq k} \cF^n(\cU),
\]
respectively.  Observe that omega and alpha limit sets of nonempty sets may be empty, but they satisfy the following properties.

\begin{proposition}
\label{prop:combomega}
Let $\cF:\cX\mvmap\cX$ be a   multivalued mapping  
and let $\cU\subset \cX$.
Then,
\begin{enumerate}
\item[(i)]
there exists a $k_*\ge 0$ such that $\bomega(\cU) = \bigcup_{n\geq k} \cF^n(\cU)$ for all $k\ge k_*$;
\item[(ii)] $\cF(\bomega(\cU)) = \bomega(\cU)$ and 
$\bomega(\cF(\cU)) = \bomega(\cU)$, and thus $\bomega(\cU) \in \IS^+(\cX)$;
\item[(iii)] $\cF$ left-total and $\cU \not = \varnothing$ implies that $\bomega(\cU)$ is invariant and $\bomega(\cU)\not = \varnothing$;
\item[(iv)] if there exists $k_*>0$  such that $\cF^n(\cU) \subset \cU$ for $k\ge k_*$,  then $\bomega(\cU) \subset \cU$;
\item[(v)] $\cV\subset \cU$ implies $\bomega(\cV) \subset \bomega(\cU)$, and in particular $\bomega(\cV \cap \cU) \subset
\bomega(\cV) \cap \bomega(\cU)$;
\item[(vi)] $\bomega(\cV\cup \cU) = \bomega(\cV) \cup \bomega(\cU)$, and in particular $\bomega(\cU) = \bigcup_{\xi\in \cU} \bomega(\xi)$;
\item[(vii)] $\bomega(\bomega(\cU)) = \bomega(\cU)$.
\end{enumerate}
The same properties hold for $\balpha$-limit sets via time-reversal, i.e. replace $\cF$ by $\cF^{-1}$.
\end{proposition}

\begin{proof}
All properties can essentially be derived from Property (i), which we prove now.
Forward images are nested sets. Since $\cX$ is finite, it follows that there exists $k_*$ such that 
\[
\bigcup_{n\geq k} \cF^n(\cU) = \bigcup_{n\geq k_*} \cF^n(\cU)
\]
for all $k\geq k_*$.
\end{proof}

\subsection{Attractors and repellers}
\label{sec:cds-AR}

%\subsection{Lattices of forward/backward invariant sets}
%\label{subsec:latt-cms}
%
Alpha and omega limit sets capture the asymptotic dynamics of individual sets. Our goal for the remainder of this section is to understand the structure of the asymptotic dynamics of all sets. We begin with the concept of forward and backward invariance.
A set $\cS\subset \cX$ is \emph{forward invariant} if $\cF(\cS) \subset \cS$ and 
it
is \emph{backward invariant} if $\cF^{-1}(\cS) \subset \cS$. The sets of forward and backward invariant sets
in $\cX$ are denoted by $\IS^+(\cX,\cF)$ and $\IS^-(\cX,\cF)$ respectively.

%\corrl I don't think that   $\IS^\pm(\cX,\cF)$ is interesting for multivalued maps (there tend to be very few such sets).
%We don't use it in the paper, so I think we should delete references to them.  KM
%<<
%Sets that are 
% both forward and backward are called \emph{forward-backward} invariant sets and are denoted by
% $\IS^\pm(\cX,\cF)$. 
% ||
% >>
%

%\corrl KM 
%<<
%\begin{proposition}
%\label{prop:cms-forback}
%The sets $\IS^-(\cX,\cF)$, $\IS^+(\cX,\cF)$, and $\IS^\pm(\cX,\cF)$ are
%finite distributive lattices with respect to intersection and union.
%The mapping $\cU\mapsto \cU^c$ is an involute lattice anti-isomorphism between $\IS^-(\cX,\cF)$ and $\IS^+(\cX,\cF)$ and
%therefore $\IS^\pm(\cX,\cF)$ is Boolean algebra with respect to set complement.
%\end{proposition}
%||
\begin{proposition}
\label{prop:cms-forback}
The sets $\IS^-(\cX,\cF)$ and $\IS^+(\cX,\cF)$ are
finite distributive lattices with respect to intersection and union.
The mapping $\cU\mapsto \cU^c$ is an involute lattice anti-isomorphism between $\IS^-(\cX,\cF)$ and $\IS^+(\cX,\cF)$. \end{proposition}

\begin{proof}
We leave the proof that  $\IS^-(\cX,\cF)$ and $\IS^+(\cX,\cF)$ are
finite distributive lattices with respect to intersection and union to the reader.

To show that set complement maps $\IS^+(\cX,\cF)$ to $\IS^-(\cX,\cF)$ consider $\cU\in \IS^+(\cX,\cF)$ and $\xi\in \cU^c$.
Suppose $\eta \in \cF^{-1}(\xi)\cap \cU$. Then 
\[
\xi \in \cF\bigl( \eta\bigr) \subset \cF(\cU) \subset \cU,
\]
which contradicts the fact that $\xi\in \cU^c$. We conclude that $\cF^{-1}(\cU^c) \subset \cU^c$, and therefore 
$\cU^c\in \IS^-(\cX,\cF)$. 
The same arguments hold when $\cU\in \IS^-(\cX,\cF)$ is backward invariant.
The fact that the map $\cU \mapsto \cU^c$ is a lattice anti-isomorphism follows from  De Morgan's laws.
\end{proof}

%\subsection{Attractors and repellers}
%\label{sec:cds-ARpairs}

To characterize the asymptotic dynamics of forward and backward invariant sets we make use of the following structures.

%\subsection{Definition and properties}
%\label{subsec:attrep12}
%Recall from \cite{KMV0}  the fundamental concepts of attractor and repeller.

\begin{definition}
\label{defn:combattractor}
Let $\cF\colon\cX\mvmap\cX$ be a combinatorial multivalued mapping.
A set $\cA \subset \cX$ is an  \emph{attractor}  for $\cF$ if 
$\cF(\cA)=\cA$. 
A set $\cR\subset \cX$ is a \emph{repeller} for $\cF$  if $\cF^{-1}(\cR) = \cR$.
 The sets of all attractors
and repellers in $\cX$ are denoted by $\sAtt(\cX,\cF)$ and $\sRep(\cX,\cF)$ respectively.
\end{definition}

For 
$\cA,\cA' \in \sAtt(\cX,\cF)$ define
\[
\cA\vee\cA' = \cA \cup \cA'\quad {\rm and}\quad\cA \wedge \cA'  = \bomega(\cA\cap\cA').
\]
Similarly, for  
$\cR,\cR' \in \sRep(\cX,\cF)$ define
$$
\cR\vee\cR' = \cR \cup \cR'\quad  {\rm and}\quad\cR \wedge \cR' = \balpha(\cR\cap\cR').
$$

\begin{proposition}
\label{prop:cms-attrep}
The sets $\bigl(\sAtt(\cX,\cF),\wedge,\vee\bigr)$ 
and $\bigl(\sRep(\cX,\cF),\wedge,\vee\bigr)$ are finite, distributive lattices.
\end{proposition}

\begin{proof}  
Let $\cA,\cA'\in \sAtt(\cX,\cF)$ be attractors. Then
  $\cF(\cA \cup \cA') = \cF(\cA) \cup \cF(\cA') = \cA \cup \cA'$, and thus $\cA\cup \cA' \in \sAtt(\cX,\cF)$.
  Similarly, $\cA\wedge \cA' = \bomega(\cA\cap \cA')$, and therefore $\cF(\cA\wedge\cA') = \cA\wedge\cA'$ by
 Proposition \ref{prop:combomega}(ii), which proves $\cA\wedge \cA' \in \sAtt(\cX,\cF)$.  
 This proves that $\sAtt(\cX,\cF)$ is a lattice. The same holds for $\sRep(\cX,\cF)$.
 It remains to show that both sublattices are distributive.

Let $\cA,\cA'\cA'' \in \sAtt(\cX,\cF)$. Then
\begin{eqnarray*}
(\cA\wedge \cA') \vee (\cA\wedge \cA'') & = &  \bomega(\cA\cap \cA') \cup \bomega(\cA\cap \cA'')\\
&= &  \bomega\bigl((\cA\cap \cA') \cup (\cA\cap \cA'')\bigr)\\
&=& \bomega\bigl(\cA\cap (\cA'\cup \cA'')\bigr) = \bomega(\cA) \wedge \bomega(\cA'\cup  \cA'')\\
&=& \cA\wedge (\cA'\vee \cA''),
\end{eqnarray*}
which proves distributivity.
The arguments for $\sRep(\cX,\cF)$ are  symmetric.
\end{proof}

\begin{remark}
Because the lattice operations are distinct,  $\sAtt(\cX,\cF)$ and $\sRep(\cX,\cF)$ cannot be viewed as sublattices  of $\IS^+(\cX,\cF)$ and $\IS^-(\cX,\cF)$, respectively.
\end{remark}

By Proposition \ref{prop:combomega}(ii) every omega limit set  is an attractor, and similarly, every alpha limit set  is a repeller.  The following result adds structure to this observation.

\begin{proposition}
\label{att-rep-fbi}
The functions 
\[\bomega\colon \IS^+(\cX,\cF) \to \sAtt(\cX,\cF)\quad \hbox{ and}\quad \balpha\colon \IS^-(\cX,\cF) \to \sRep(\cX,\cF),
\]
are lattice epimomorphisms.
\end{proposition}

\begin{proof}
We give the proof for $\bomega\colon \IS^+(\cX,\cF) \to \sAtt(\cX,\cF)$.
From Proposition \ref{prop:combomega}(vi) it follows that $\bomega(\cU\cup \cU') = \bomega(\cU) \cup \bomega(\cU')$.  For the meet operation we argue as follows.
Since $\cU$ and $\cU'$ are forward invariant, we have  $\cA = \bomega(\cU)\subset \cU$ and $\cA' = \bomega(\cU')\subset \cU'$. Then
\begin{eqnarray*}
\cA \wedge \cA' &=& \bomega(\cA \cap \cA') \subset \bomega(\cU\cap \cU') = \bomega\bigl( \bomega(\cU\cap \cU')\bigr)\\
 &\subset& \bomega\bigl( \bomega(\cU)\cap \bomega(\cU')\bigr) = \cA\wedge \cA',
\end{eqnarray*}
which proves that $\bomega(\cU\cap \cU') = \cA \wedge \cA' = \bomega(\cU) \wedge \bomega(\cU')$, and
therefore $\bomega$ is a lattice homomorphism. The map is surjective, since $\sAtt(\cX,\cF) \subset \IS^+(\cX,\cF)$ and $\bomega|_{\sAtt(\cX,\cF)} = \id$.
Moreover,  $\bomega(\varnothing) = \varnothing = 0$ in $\sAtt(\cX,\cF)$, and $\bomega(\cX,\cF) = 1$ in $\sAtt(\cX,\cF)$.
\end{proof}

%\subsection{Attracting and repelling sets}
%\label{subsec:attandrep}

While the previous proposition demonstrates that all attractors and repellers can be obtained as  the omega and alpha limit sets of a forward and backward invariant sets, there are larger collections of sets that lead to attractors and repellers. 
An \emph{attracting set} $\cU$ has the property that $\bomega(\cU) \subset \cU$, and the attracting
sets are denoted by $\sASet(\cX,\cF)$. Similarly, a \emph{repelling set} is defined by $\balpha(\cU) \subset \cU$, and the repelling sets are denoted by
$\sRSet(\cX,\cF)$. 
Observe that forward and backward invariant sets are attracting and repelling sets, respectively, but not vice-versa.

\begin{proposition}
The sets $\sASet(\cX,\cF)$ and $\sRSet(\cX,\cF)$ are finite  sublattices of $\sSet(\cX,\cF)$ and
therefore finite distributive lattices. 
\end{proposition}

\begin{proof}
The proof for $\sASet(\cX,\cF)$ follows from the following containments
\[
\begin{aligned}
\bomega(\cU\cup\cU') &= \bomega(\cU) \cup \bomega(\cU') \subset \cU\cup \cU',\\
\bomega(\cU\cap\cU') &\subset \bomega(\cU) \cap \bomega(\cU') \subset \cU\cap \cU'.
\end{aligned}
\]
The proof for $\sRSet(\cX,\cF)$ is similar.
\end{proof}

The same proof as that of Proposition~\ref{att-rep-fbi} leads to the following result.
\begin{proposition}
\label{att-rep-aset}
The mappings $\bomega: \sASet(\cX,\cF) \to \sAtt(\cX,\cF)$ and $\balpha: \sRSet(\cX,\cF) \to \sRep(\cX,\cF)$  are lattice epimorphisms.
\end{proposition}

Proposition \ref{prop:cms-forback} establishes  $\cU \mapsto \cU^c$ as a lattice anti-isomorphism between $\IS^+(\cX,\cF)$ and
$\IS^-(\cX,\cF)$. The result is true for  attracting and repelling sets. To prove this we make use of the following 
result.

\begin{proposition}
\label{prop:char-attset}
A set $\cU$ is an attracting set if and only if there exists $k>0$ such that $\cF^n(\cU) \subset \cU$ for all $n\ge k$. 
Similarly, a set $\cU$ is a repelling set if and only if there exists 
$k\ge 0$ such that $\cF^{-n}(\cU) \subset \cU$ for all $n\ge k$.
\end{proposition}

\begin{proof}
If $\cU \in \sASet(\cX,\cF)$, then $\bomega(\cU) \subset \cU$. By Proposition \ref{prop:combomega}(i), there exists
$k>0$ such that $\bomega(\cU) = 
\Gamma^+_k(\cU) = \bigcup_{n\ge k} \cF^n(\cU) \subset \cU$, which implies that $\cF^n(\cU) \subset \cU$ for all
$n\ge k$.

Conversely, if there exists $k>0$ such that $\cF^n(\cU) \subset \cU$ for all
$n\ge k$, then Proposition \ref{prop:combomega}(iv) implies that $\bomega(\cU) 
%= \cF^n(\cU) 
\subset \cU$,
which proves that $\cU\in \sASet(\cX,\cF)$.
\end{proof}

\begin{proposition}
\label{prop:invo}
The mapping $\cU\mapsto \cU^c$ is a lattice anti-isomorphism between $\sASet(\cX,\cF)$ and
$\sRSet(\cX,\cF)$.
\end{proposition}

\begin{proof}
Let $\cU\in \sASet(\cX,\cF)$. Then by Proposition \ref{prop:char-attset}, there exists $k>0$ such that  $\cF^n(\cU) \subset \cU$ for all $n\ge k$.
As in the proof of Proposition \ref{prop:cms-forback}, assume
$\xi\in \cU^c$.
Suppose that there exists $k>0$ such that  $ \cF^{-n}(\xi)\cap \cU\neq \varnothing$ for all $n\ge k$.
Let $\eta\in\cF^{-n}(\xi)\cap \cU$ for $n\ge k$. Then %by Lemma \ref{lem:group-prop}
we have
$$
\xi \in \cF^n\bigl( \eta\bigr) \subset \cF^n(\cU) \subset \cU,
$$
which contradicts the fact that $\xi\in \cU^c$. We conclude that there exists $k>0$ such that $\cF^{-n}(\cU^c)
\subset \cU^c$ for all $n\ge k$, and therefore $\cU^c \in \sRSet(\cX,\cF)$.
\end{proof}

\subsection{Attractor-repeller pairs}
\label{subsec:attrep-pairs}
For a multivalued mapping $\cF\colon \cX\mvmap\cX$  one can introduce the notions of dual repeller and dual attractor.

\begin{definition}
\label{defn:dual-cms}
Let  $\cF\colon \cX\mvmap\cX$ be a multivalued mapping. The  \emph{dual repeller}
$\cA^*$ to an attractor $\cA$ is 
defined by  $\cA^* = \balpha(\cA^c)$.
%$\cA^* = \balpha (\cA^c)$.
Similarly the {\em dual attractor} to a repeller $\cR$ is $\cR^* = \bomega(\cR^c)$.
%$\cR^*=\bomega(\cR^c)$.
The pairs $(\cA,\cA^*)$ and $(\cR^*,\cR)$ are called \emph{attractor-repellers pairs} in $\cX$.
\end{definition}

It follows from Proposition \ref{prop:cms-forback} and Proposition \ref{prop:combomega}(vii) that if $\cA\in\sAtt(\cX,\cF)$, then
$\cA^c \in \IS^-(\cX,\cF)$ and thus $\cA^* = \balpha(\cA^c)\in\sRep(\cX,\cF)$.
Similarly, if $\cR\in\sRep(\cX,\cF)$, then $\cR^* = \bomega(\cR^c)\in \sAtt(\cX,\cF)$.
%\corrl Do we use this fact? If not let's remove. KM
%<<
%Since $\cA$ and $\cA^c$ are disjoint, and $\cA^c \in \IS^-(\cX,\cF)$, it follows that $\cA\cap \cA^* = \varnothing$.
%||
%>>

\begin{proposition}
\label{prop:cms-star}
%Let $\cF:\cX\mvmap\cX$ be a closed combinatorial multivalued mapping and 
Let $(\cA,\cA^*)$ 
be a attractor-repeller pair. Then,
$$
\cA = \bomega({\cA^*}^c)\hbox{~and~} \cA^* =  \balpha(\cA^c).
$$
The operator $\cA \mapsto \cA^*$ is a lattice anti-isomorphism from $\sAtt(\cX,\cF)$ to $\sRep(\cX,\cF)$.
\end{proposition}

\begin{proof}
The set $\cA^{**}  = \bomega\bigl((\cA^*)^c\bigr)$ is forward-backward invariant, and hence for
every $\xi \in \cA^{**}$ we have $\balpha(\xi)\subset \cA^{**}$ and  $\bomega(\xi)\subset \cA^{**}$.
Moreover, $\cA = (\cA^c)^c  \subset (\cA^*)^c$, and thus
$\cA = \bomega(\cA) \subset \bomega\bigl((\cA^*)^c\bigr) = \cA^{**}$.
Let $\xi \in \cA^{**}\setminus \cA = \cA^{**} \cap \cA^c$. Since $\cA^c$ is backward invariant,
$\balpha(\xi) \subset  \cA^c$, and thus $\balpha(\xi) \in \cA^{**} \setminus \cA$.

Also 
$\balpha(\xi) \subset \balpha(\cA^c) = \cA^*$,
which implies that $\balpha(\xi) \in \cA^{**} \cap \cA^*$. The forward invariance of $(\cA^*)^c$ implies that
$\cA^{**}\cap \cA^* = \varnothing$. We conclude that that $\balpha(\xi) = \varnothing$ for all $\xi\in \cA^{**} \setminus \cA$.
By definition $\cA^{**}$ is an attractor, and therefore $\cF^{-1}(\xi) \cap \cA^{**} \not = \varnothing$ for all $\xi\in \cA^{**}$,
and consequently $\balpha(\xi) \not = \varnothing$ for all $\xi \in \cA^{**}$, a contradiction.
This shows that $\cA^{**} = \cA$. Similar arguments 
also apply to repellers. 
The mapping $\cA\mapsto \cA^*$ is an involution, and the lattices $\sAtt(\cX,\cF)$ to $\sRep(\cX,\cF)$ are isomorphic.

To show that $\cA \mapsto \cA^*$ is a lattice anti-isomorphism we argue as follows.
Let $\cA^* = \balpha(\cA^c)$ and  $\cA'^* = \balpha({\cA'}^c)$, then by Proposition \ref{att-rep-fbi}
and De Morgans' laws
\begin{eqnarray*}
(\cA \vee \cA')^* &=& \balpha\bigl((\cA\vee\cA')^c\bigr) = \balpha\bigl((\cA\cup\cA')^c\bigr)\\
&=& \balpha\bigl(\cA^c\cap {\cA'}^c\bigr) = \balpha(\cA^c)\wedge \balpha({\cA'}^c)\\
&=& \cA^* \wedge \cA'^*.
\end{eqnarray*}
The same holds for $\cA^*$ and $\cA'^*$, i.e. $(\cA^* \vee \cA'^*)^* = \cA^{**} \wedge \cA'^{**}
= \cA \wedge \cA'$.
Observe that 
$$
(\cA \wedge \cA')^* =  (\cA^* \vee \cA'^*)^{**} =  \cA^* \vee \cA'^*,
$$
which proves the proposition.
\end{proof}

Much of the discussion of this section up to this point can be summarized in the following commutative diagram of lattice homomorphisms.
\begin{equation}
\label{eq:diag+-}
\begin{diagram}
\node{\IS^+(\cX,\cF)}\arrow{se,l,V}{\imath}\arrow{sse,l,A}{\bomega} \arrow[3]{e,l,<>}{^c}  \node{}\node{}
\node{\IS^-(\cX,\cF)}\arrow{ssw,r,A}{\balpha}\arrow{sw,l,V}{\imath}\\
\node{}\node{\sASet(\cX,\cF)} \arrow{e,l,<>}{^c} \arrow{s,l,A}{\bomega}\node{\sRSet(\cX,\cF)} \arrow{s,r,A}{\balpha}\\
\node{}\node{\sAtt(\cX,\cF)} \arrow{e,l,<>}{^*}   \node{\sRep(\cX,\cF).} 
\end{diagram}
\end{equation}

% !TEX root = ./attractors-IIs2.tex

\section{From continuous  dynamics  to multivalued mappings}
\label{sec:cont-graph}
In this section we  recall how the dynamics of multivalued mappings can be linked to dynamical systems
as described in \cite{KMV0}.
Since multivalued mappings are discrete in both time and space, we need to address the issues of both time and
space discretization.

\subsection{Grids and outer approximations}
\label{sec:cgt}
To represent  continuous dynamics in terms of the combinatorial structures described in Section~\ref{sec:comb-syst}
requires discretizing phase space.  We wish this discretization to be as generally applicable and as topologically nice as possible. With this in mind we choose the basic elements of our discretization to be {\em regular closed sets}, i.e.
sets $A\subset X$ such that $A =\cl(\Int(A))$.

\begin{proposition}
\label{prop:walker} \cite[Proposition 2.3]{Walker}
Let $X$ be a topological space. The family $\scrR(X)$ of regular closed subsets of $X$ is a  Boolean algebra
with the following operations:
\begin{enumerate}
\item[(i)] $A\leq B$ if and only if $A\subset B$;
\item[(ii)]  $A\vee B := A\cup B$; %$\bigvee_\alpha A_\alpha := \cl\left( \bigcup_\alpha \Int(A_\alpha) \right)$;
\item[(iii)]  $A\wedge B := \cl ( \Int(A\cap B))$; %$\bigwedge_\alpha A_\alpha := \cl\left( \Int(\bigcap_\alpha A_\alpha) \right)$;
\item[(iv)] $A^\rc := \cl(A^c)$;
\end{enumerate}
where $0=\varnothing$ and $1=X$.
\end{proposition}

\begin{remark}
Proposition 2.3 in \cite{Walker} proves that $\scrR(X)$ is a complete Boolean algebra, i.e. 
$\bigvee_\alpha A_\alpha := \cl\left( \bigcup_\alpha \Int(A_\alpha) \right)$ and $\bigwedge_\alpha A_\alpha := \cl\left( \Int(\bigcap_\alpha A_\alpha) \right)$ are well-defined.
\end{remark}

\begin{lemma}
\label{lem:equivregcl}
Let $A,A' \in \scrR(X)$, then $A\wedge A' = \varnothing$ if and only if $A\cap \Int (A') = \varnothing$.
\end{lemma}
\begin{proof}
By definition
\[
 A \wedge A' = \cl(\Int(A\cap A')).
\]
Using the property $\Int(A\cap A') = \Int(A)\cap \Int(A')$, 
\[
 A \wedge A' =  \cl(\Int(A)\cap \Int(A')).
\]
Also, if $U\subset X$ is open and $B,B'\subset X$ with $\cl(B) = \cl(B')$, then 
$\cl(B\cap U) = \cl(B'\cap U)$. 
Taking $U= \Int(A')$, $B = \Int(A),$ and $B'= A$ implies
\[
 A \wedge A' = \cl(A\cap \Int(A')).
\]
Therefore
\[
A \wedge A'=\varnothing \quad\hbox{iff}\quad  \cl(A\cap \Int(A'))=\varnothing
\quad\hbox{iff}\quad A\cap \Int(A')=\varnothing,
\]
which  proves the equivalence.
\end{proof}

% \corrc Note the new terminology here - WK 7/16/14 <<>>

Sets $A,A'\subset X$ for which $A\wedge A' = \varnothing$ will be  referred to as \emph{regularly disjoint sets}.

\begin{lemma}
\label{lem:intregdisj}
Let $A,B,C \in \scrR(X)$ be mutually regularly disjoint sets. Then
\[
A = \cl\bigl( (A\cup B)\setminus (B\cup C)\bigr).
\]
\end{lemma}
\begin{proof}
We start with the observation that if $A,A'\in \scrR(X)$ are mutually regularly disjoint, then $\cl(A\setminus A') = A$. Indeed, by Lemma \ref{lem:equivregcl}, $A\wedge A'=\varnothing$, is equivalent to $A\cap \Int(A')
=\Int(A) \cap A' = \varnothing$. This implies $\Int(A) \subset A\setminus A' \subset A$ and therefore
$A= \cl(\Int(A)) \subset \cl(A\setminus A') \subset \cl(A) = A$, which proves the statement.

Note that $(A\cup B)\setminus (B\cup C) = (A\cup B) \setminus B\setminus C = A\setminus B\setminus C
= A\setminus (B\cup C)$.
By assumption $A\wedge (B\cup C) = \varnothing$. By the previous statement we then have
\[
A = \cl (A\wedge (B\cup C)) =  \cl\bigl( (A\cup B)\setminus (B\cup C)\bigr),
\]
which proves the lemma.
\end{proof}

For the purpose of computation we are only interested in finite collections of regular closed sets.

%\corrc I went back to the original notation because the atoms are grid elements not the indexing set for the grid elements. -- WK 7/16/14 <<>>

\begin{proposition}
\label{prop:grid-sub}
Let   $\scrR_0\subset \scrR(X)$ be a finite subalgebra of the Boolean algebra $\scrR(X)$ of regular closed subsets of $X$ and let $\sJ(\scrR_0)$ denote the set of atoms of $\scrR_0$.  Then
\begin{enumerate}
\item[(i)] $X=\bigcup\setof{A\mid A \in \sJ(\scrR_0)}$.
\item[(ii)]
If $A,A'\subset X$   are atoms of $\scrR_0$, then
$
A\cap \Int(A') = \varnothing.
$
\end{enumerate}
Conversely, every finite set $\sJ= \{A\mid A\subset X\}$ of mutually regularly disjoint  subsets which satisfies (i)  
generates a subalgebra of $\scrR(X)$ for which $\sJ$ is the set of atoms.
\end{proposition}
\begin{proof}
The proof of (i) follows from the fact that $1=X$ and Proposition~\ref{prop:walker}(ii). Property (ii) follows from Lemma \ref{lem:equivregcl}. The converse statement follows from Stone's Representation Theorem, cf.\ \cite{Davey:2002p5923}.
\end{proof}

%\corrl
%For now I left the definition of $\sGrid(X)$, RC 7/21/14
%<<
%Proposition~\ref{prop:grid-sub} implies that if $X$ is a compact metric space, then any finite subalgebra of $\scrR(X)$ defines a {\em grid} on $X$ (see \cite{KMV0}, \cite{Mrozek-grid}), and conversely a grid defines a finite subalgebra of $\scrR(X)$.  
%Since we make use of grids to pass from the computations to dynamics, we recall and establish several fundamental properties.
%First, by \cite[Theorem 2.2]{KMV0} given a compact metric space there exists a grid with elements of arbitrarily small diameter. 
%Second, as is discussed in detail in this section, grids provide a natural correspondence between the combinatorial systems of Section~\ref{sec:comb-syst} and the continuous systems of interest. 
%||
Proposition~\ref{prop:grid-sub} implies that if $X$ is a compact metric space, then any finite subalgebra of $\scrR(X)$ defines a {\em grid} on $X$ (see \cite{KMV0}, \cite{Mrozek-grid}), and conversely a grid defines a finite subalgebra of $\scrR(X)$.  
We denote the space of a grids on $X$ by $\sGrid(X)$, which is a lattice dual to the lattice of finite subalgebras
$\subFR(X)$.
Since we make use of grids to pass from the computations to dynamics, we recall and establish several fundamental properties.
First, by \cite[Theorem 2.2]{KMV0} given a compact metric space there exists a grid with elements of arbitrarily small diameter. 
Second, as is discussed in detail in this section, grids provide a natural correspondence between the combinatorial systems of Section~\ref{sec:comb-syst} and the continuous systems of interest.

To begin to set up the relationship between combinatorial and continuous systems, consider a grid  on $X$ indexed by a finite set $\cX$.  
In particular, given $\xi\in\cX$  the corresponding grid element is denoted by $\supp{\xi}\in \scrR(X)$.  
The \emph{evaluation mapping} $\supp{\cdot}\colon \sSet(\cX) \to \scrR(X)$ is defined by
\[ 
 |\cU| := \bigcup_{\xi\in \cU} \supp{\xi}.
 \]
 The range of $\sSet(\cX)$ under $|\cdot|$ is the subalgebra whose atoms are
 the grid elements, and this subalgebra will be denoted by $\scrR_{{\cX}}(X)$.
 Proposition~\ref{prop:grid-sub} immediately implies the following.
\begin{corollary}
\label{cor:grid->lattice}\label{prop:grid->lattice}
Given a grid on a compact metric space $X$ indexed by $\cX$, then the evaluation mapping $\supp{\cdot}\colon \sSet(\cX) \to \scrR(X)$ is a 
Boolean isomorphism onto $\scrR_{{\cX}}(X)$.
\end{corollary}
\begin{proof}
%It is immediate that $|\cU \cup \cU'| = |\cU| \cup |\cU'|  = |\cU|\vee |\cU'|$. In order to prove the identity $|\cU \cap \cU'| = |\cU|\wedge |\cU'|$, we argue as follows.
%Write $\cU = \cV \sqcup \cW$ and $\cU' = \cV \sqcup \cW'$ and $\cW\cap \cW' = \varnothing$.
%Then, $\cU \cap \cU' = \cV$, and 
%\[
%|\cU|\wedge |\cU'|  = (|\cV|\vee |\cW|) \wedge (|\cV|\vee |\cW'|) = |\cV| \vee  (|\cW|\wedge |\cW'|) = |\cV| = |\cU\cap \cU'|,
%\]
%since $|\cW|\wedge |\cW'| =\varnothing$ since all grid elements in $\cX$ are disjoint with respect to $\wedge$.
%%
By construction the evaluation map is a lattice homomorphism.
Since $|\varnothing| = \varnothing$ and $|\cX| = X$, Lemma 4.17 in  \cite{Davey:2002p5923}  shows that the evaluation map is Boolean.
%, which states that a lattice homomorphism between Boolean algebras is 
%a Boolean homomorphism if and only if the $0$ and $1$ elements are preserved.
\end{proof}

Before explicitly describing the discretization of a general dynamical system $\varphi\colon \T^+ \times X\to X$, we consider the simple setting of approximating the dynamics generated by a continuous map $f\colon X\to X$.

\begin{definition}
\label{defn:outerapprox}
Let $f\colon X \to X$ be a continuous map. Let $\cX$ be the indexing set for a grid on $X$.
A multivalued mapping $\cF\colon\cX \mvmap \cX$ is an \emph{outer approximation} of $f$ if
\begin{equation}
\label{eq:outerapprox}
 f(\supp{\xi}) \subset \Int |\cF(\xi)| \text{ for all }\xi\in \cX.
\end{equation}
A multivalued mapping $\cF\colon\cX \mvmap \cX$ is a \emph{weak outer approximation} for 
 $f$ if
\begin{equation}
\label{eq:weakouterapprox}
f(\supp{\xi}) \subset \Int \supp{\bigcup_{n\geq 0}\cF^n(\xi)}\text{ for all }\xi\in \cX.
\end{equation}
\end{definition}

%\corrl
%Remark about calling the indexing set $\cX$ a grid, RC 7/22/14
%<<
%||
%\begin{remark}
%We will abuse terminology and refer to the indexing set $\cX$   as a grid on $X$.
%\end{remark}
%>>

\begin{remark}
By  definition  outer approximations $\cF$  are necessarily left-total, and therefore combinatorial omega limit sets and attractors are invariant sets for outer approximations $\cF$.
\end{remark}

\subsection{Attractors, repellers, and their neighborhoods}
\label{subsec:at-rep-neigh}

Recall that a set $U\subset X$ is an \emph{attracting neighborhood} for a continuous function $f\colon X \to X$ if
$\omega(U,f)\subset \Int(U)$.  A \emph{trapping region} $U$ is an attracting neighborhood with the additional 
property that $f(\cl(U)) \subset \Int(U)$.  A set $A\subset X$ is an \emph{attractor}  if there exists a 
trapping region $U$  such that $A=\Inv (U,f)$ in which case $A = \omega(U,f)\subset \Int(U)$.

A set $U\subset X$ is a \emph{repelling neighborhood} for a continuous function $f\colon X \to X$ if
$\alpha(U,f)\subset \Int(U)$.  A \emph{repelling region} $U$ is an repelling neighborhood with the additional 
property that $f^{-1}(\cl(U)) \subset \Int(U)$.  A set $R\subset X$ is an \emph{repeller}  if there exists a 
repelling region $U$  such that $R=\Inv^+ (U,f)$ in which case $R = \alpha(U,f)\subset \Int(U)$ cf.\ \cite{KMV-1a}.

The sets of all attracting neighborhoods and repelling neighborhoods are denoted by $\sANbhd(X,f)$ and 
$\sRNbhd(X,f)$ respectively. As is shown in \cite{KMV-1a} these sets are lattices under the operations union and intersection.
The following propositions indicate that attracting and repelling neighborhoods can be identified using weak outer
approximations.

\begin{proposition}
\label{prop:discr-att-to-varphi1}
Let $\cF\colon \cX \mvmap \cX$ be an weak outer approximation for $f$. If $\cU\subset \sInvset^+(\cX,\cF)$, then $|\cU|$ is  a trapping region for $f$, and therefore $|\cU|\in\sANbhd(X,f)$. %is an attracting neighborhood. 
%, and consequently an attracting neighborhood for  $\varphi$, i.e. $|\cU| \in \sANbhd(X,\varphi)$.
\end{proposition}
\begin{proof} Since $\cF$ is a weak outer approximation,
for $\xi\in\cU$ we have
\[
f(|\xi|) \subset \Int\left|\bigcup_{n\ge0}\cF^n(\xi)\right| \subset \Int|\cU|
\]
because $\cF^n(\cU)\subset \cU$ for all $n\ge0$.
Therefore
$
f(|\cU|)  \subset \Int|\cU|,
$
which implies that $|\cU|$ is a trapping region for $f$. %, and thus an attracting
%neighborhood.
%By Lemma \ref{lem:att-aux1} we have $\omega(|\cU|,\varphi) \subset \Int|\cU|$,   proving that $U= |\cU|$ is an attracting
%neighborhood for $\varphi$.
\end{proof}

\begin{proposition}
\label{prop:prop:discr-att-to-varphi2}
Let $\cF\colon\cX \mvmap \cX$ be an weak outer approximation for $f$. If $\cU\in\IS^-(\cX,\cF)$, then   $|\cU|\in \sRNbhd(X,f)$.
\end{proposition}
\begin{proof}
Let $\cU\in\IS^-(\cX,\cF)$. 
By Proposition~\ref{prop:cms-forback}, $\cU^c\in\IS^+(\cX,\cF)$, and thus by Proposition~\ref{prop:discr-att-to-varphi1}, $\supp{\cU^c}\in \sANbhd(X,f)$.  
By \cite[Corollary 3.24]{KMV-1a} $\supp{\cU^c}^c \in \sRNbhd(X,f)$ and thus by \cite[Corollary 3.26]{KMV-1a}
$\cl\left(\supp{\cU^c}^c\right) \in \sRNbhd(X,f)$.  Finally, by Corollary~\ref{prop:grid->lattice}, $|\cU|\in \sRNbhd(X,f)$.
\end{proof}

\begin{remark}
\label{rmk:other1}
Another approach is to achieve the latter directly. In that case a negative time variation on (\ref{eq:weakouterapprox}) is needed.
The duality approach used here does not require additional assumptions and is therefore preferable.
However, the duality approach implies that $|\cU|$ is a repelling neighborhood, and it is \emph{not} clear whether $|\cU|$ is a
repelling \emph{region}.
\end{remark}

As the following proposition indicates, outer approximations, as opposed to weak outer approximations, allow one to obtain the same results using a larger variety of sets. 
The proof makes use of the following  observation.  If $\cF$ is an outer approximation for $f$, then 

\begin{equation}
\label{eqn:iterate}
f^n(|\xi|)  \subset \Int |\cF^n(\xi)|\quad \forall \xi\in \cX,\; \forall n\ge 0.
\end{equation}
This property can be derived as follows. Observe that 
\[
f^2(|\xi|) =  f(f(|\xi|))\subset f(\Int |\cF(\xi)|) \subset f(|\cF(\xi)|), 
\]
and by definition 
\[
f(|\cF(\xi)|) = \bigcup_{\xi' \in\cF(\xi)} f(|\xi'|) \subset \bigcup_{\xi' \in\cF(\xi)} \Int \supp{\cF(\xi')}
\subset \Int \supp{\cF^2(\xi)}.
\]
Equation \eqref{eqn:iterate} follows by proceeding inductively.

\begin{proposition}
\label{prop:discr-att-to-varphi2}
Let $\cF\colon\cG \mvmap \cG$ be an  outer approximation for $f$. If  
$\cU\in \sASet(\cX,\cF)$, then $|\cU| \in \sANbhd(X,f)$.
If $\cU\in \sRSet(\cX,\cF)$, then $|\cU| \in \sRNbhd(X,f)$.
\end{proposition}

\begin{proof}
By Proposition~\ref{prop:char-attset} if $\cU\in \sASet(\cX,\cF)$, then there exists an $n\ge 1$ such that $\cF^k(\cU) \subset \cU$ for all $k\ge n$.
By \eqref{eqn:iterate}
$f^k(\xi) \subset \Int |\cF^k(\xi)|$ for all $k\ge n$ and all $\xi\in \cU$.
Therefore,
\begin{eqnarray*}
f^k(\supp{\cU}) &=&  \bigcup_{\xi\in \cU} f^k(\supp{\xi}) \subset \bigcup_{\xi\in \cU} \Int \bigl(\supp{\cF^k(\xi)}\bigr)\\
&\subset& \Int \left[ \bigcup_{\xi\in \cU}\supp{\cF^k(\xi)} \right] = \Int \supp{ \cF^k(\cU)} \subset  \Int |\cU|\;\;\forall k\ge n,
\end{eqnarray*}
and hence $\supp{\cU}\in \sANbhd(X,f)$.

To prove the second part, let $\cU\in \sRSet(\cX,\cF)$.  By Proposition~\ref{prop:invo} $\cU^c\in \sASet(\cX,\cF)$ and
thus $|\cU^c| \in \sANbhd(X,f)$.  By Corollary~\ref{prop:grid->lattice} 
\[
|\cU|^\rc = \supp{\cU^c},
\]
 and thus by \cite[Corollary 3.26]{KMV-1a} $|\cU| \in \sRNbhd(X,f)$.
\end{proof}

Let $\sANbhdR(X,f)$ and $\sRNbhdR(X,f)$ denote the sets of regular closed attracting and repelling neighborhoods, respectively.
\begin{proposition}
\label{prop:ANbhdR}
Given a continuous function $f\colon X\to X$ on a compact metric space, $\sANbhdR(X,f)$ and $\sRNbhdR(X,f)$ are sublattices of $\scrR(X)$.
\end{proposition}

\begin{proof}
Since the elements of $\sANbhdR(X,f)$ are regular closed sets, then $\sANbhdR(X,f)\subset \scrR(X)$. 
Thus it only needs to be shown that $\sANbhdR(X,f)$ is a bounded lattice. 
Let $U,U' \in \sANbhdR(X,f)$.
Observe that $\varnothing, X\in \sANbhdR(X,f)$. 
By \cite[Lemma 3.2]{KMV-1a}, $\Int(U) \cap \Int(U') \in \sANbhdR(X,f)$, and therefore
\[
U\wedge U' = \cl \bigl( \Int(U) \cap \Int(U') \bigr) \in \sANbhd(X,f),
\]
which
proves that $\sANbhdR(X,f)$ is closed under the operations $\vee$ and $\wedge$ of $\scrR(X)$.
The same argument applies to $\sRNbhdR(X,f)$.
\end{proof}

\begin{remark}\label{rmk:eval}
As indicated in \cite{KMV-1a} $\sANbhd(X,f)$  is a sublattice of $\sSet(X)$ and  by Proposition~\ref{prop:ANbhdR}
$\sANbhdR(X,f)$ is a sublattice of $\scrR(X)$.  Since the operations in these two lattices are different $\sANbhd(X,f)$
and $\sANbhdR(X,f)$ are not interchangeable. The same comment applies to $\sRNbhd(X,f)$ and $\sRNbhdR(X,f)$.
\end{remark}

\begin{theorem}
\label{thm:first-evalf}
Let $\cX$ be an indexing set for a grid on $X$ and let $\cF\colon\cX\mvmap \cX$ be a weak outer approximation of $f$. Then,
\begin{equation}\label{diag:AR2}
\begin{diagram}
\node{\IS^+(\cX,\cF)} \arrow{e,l,<>}{^c} \arrow{s,l,V}{ |\cdot|}\node{\IS^-(\cX,\cF)} \arrow{s,r,V}{ |\cdot|}\\
\node{\sANbhdR(X,f)} \arrow{e,l,<>}{^\rc} \arrow{s,l,A}{\displaystyle \omega}\node{\sRNbhdR(X,f)} \arrow{s,r,A}{\displaystyle \alpha}\\
\node{\sAtt(X,f)} \arrow{e,l,<>}{^*}   \node{\sRep(X,f)} 
\end{diagram}
\end{equation}
is a commuting diagram of distributive lattices. 
The same statement holds if we replace $\IS^+(\cX,\cF)$ and $\IS^-(\cX,\cF)$ by $\sASet(\cX,\cF)$ and
$\sRSet(\cX,\cF)$ respectively.
\end{theorem}

The proof of this result makes use of the following lemma.

\begin{lemma}
\label{lem:omega-int}
Let $U\in \sANbhd(X,f)$, then $\omega(U) = \omega(\cl(U)) = \omega(\Int(U)) \subset \Int(U)$.
\end{lemma}

\begin{proof}
If $U$ is attracting, then  $\Int(U)$ is attracting. 
This implies $\omega(\Int(U)) = \Inv(\Int(U))$, see \cite[Corollary 3.6]{KMV-1a}.
Moreover, $\omega(\Int(U)) \subset \omega(U) = \Inv(U,f) \subset \Int(U)$, which implies
that $\omega(\Int(U)) = \omega(U)$.
\end{proof}

\begin{proof}[Proof of Theorem~\ref{thm:first-evalf}]
The proof of the upper square follows from Propositions~\ref{prop:discr-att-to-varphi2} and \ref{prop:ANbhdR} and
the relation $|\cU|^\rc = |\cU^c|$.

The first step in the proof of the lower square is to show that $\omega\colon \sANbhdR(X,f) \to \sAtt(X,f)$, and
$\alpha\colon \sRNbhdR(X,f) \to \sRep(X,f)$ are lattice homomorphisms.
For $\vee = \cup$ the homomorphism property is obvious. 
As for $\wedge$, applying Lemma~\ref{lem:omega-int} and \cite[Proposition 4.1]{KMV-1a} results in
\[
\omega(U\wedge U') = \omega\bigl(\cl(\Int(U\cap U'))\bigr)  = \omega(\Int(U\cap U')) =\omega(U\cap U') = A\wedge A'.
\]
The same argument applies to repelling neighborhoods.
The surjectivity of $\alpha$ and $\omega$ follows from \cite[Proposition 5.5]{KMV0}
\end{proof}

Strengthening the assumptions on the outer approximation $\cF$ allows one to extend Theorem~\ref{thm:first-evalf}.

\begin{theorem}
\label{thm:commutative-diagram}
Let $\cX$ be an indexing set for a grid on $X$ and let $\cF\colon\cX\mvmap \cX$ be an outer approximation of $f$. 
Then
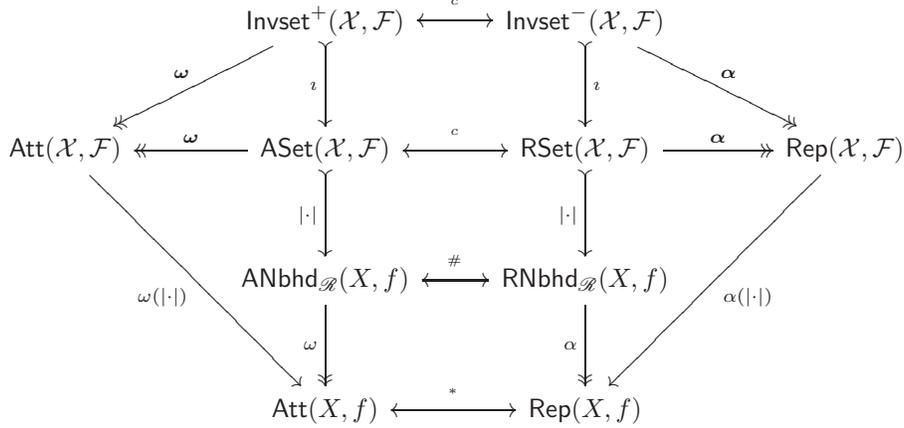
\begin{equation}
\label{diag:comm-liftfull}
\begin{diagram}
\node{}\node{\IS^+(\cX,\cF)} \arrow{e,l,<>}{^c} \arrow{s,l,V}{\imath} \arrow{sw,l,A}{\bomega}\node{\IS^-(\cX,\cF)} \arrow{s,r,V}{\imath}\arrow{se,l,A}{\balpha}	\\
\node{\sAtt(\cX,\cF)}\arrow{sse,r}{\omega(|\cdot|)}\node{\sASet(\cX,\cF)}\arrow{s,l,V}{|\cdot|}\arrow{w,l,A}{\bomega}\arrow{e,l,<>}{^c} 
\node{\sRSet(\cX,\cF)}\arrow{s,l,V}{|\cdot|}\arrow{e,l,A}{\balpha}\node{\sRep(\cX,\cF)} \arrow{ssw,r}{\alpha(|\cdot|)}	\\
\node{}\node{\sANbhdR(X,f)}\arrow{s,l,A}{\omega}\arrow{e,l,<>}{^\rc} \node{\sRNbhdR(X,f)}\arrow{s,l,A}{\alpha} \\
\node{}\node{\sAtt(X,f)} \arrow{e,l,<>}{^*} \node{\sRep(X,f)} 
\end{diagram}
\end{equation}
is a commutative diagram of distributive lattices where $\imath$ denotes  inclusion.
\end{theorem}

The proof of Theorem~\ref{thm:commutative-diagram} follows directly from Theorem~\ref{thm:first-evalf}, the following two lemmas, and the use of  duality between attractors and repellers to obtain the same lemmas for repellers.

\begin{lemma}
\label{lem:commute1}\label{prop:commute1}
Let $\cF\colon\cX\mvmap \cX$ be an outer approximation for a continuous mapping $f\colon X \to X$, and
let $\cU\in \sASet(\cX,\cF)$. Then $\omega(|\cU|) = \omega(|{\bomega}(\cU)|)$.
\end{lemma}

\begin{proof}
From Proposition \ref{prop:combomega}(iv) we have that 
$\bomega(\cU) = \cF^k(\cU)$, for some $k$ large enough.
Set $S =  \omega(|\cU|,f)$. Then $f(S) = S$, and
\[
S = f^k(S) \subset f^k(|\cU|) \subset |\cF^k(\cU)| = |\bomega(\cU)|,
\]
which proves the lemma.
\end{proof}

\begin{lemma}
\label{lem:latt-latt2}
The mapping  $\omega(|\cdot|)\colon\sAtt(\cX,\cF) \to \sAtt(X,f)$ is a lattice homomorphism.
\end{lemma}

\begin{proof}
The property for $\vee = \cup$ is obvious. Therefore, we restrict the proof to $\wedge$.
Let $\cA,\cA'\in \sAtt(\cX,\cF)$.
We have that 
\[
\omega(|\cA\wedge \cA'|) \subset \omega(|\cA\cap \cA'|) = \omega(|\cA|\wedge|\cA'|) = \omega(|\cA|) \wedge \omega(|\cA'|)
=A\wedge A',
\]
where $A = \omega(|\cA|)$ and $A' = \omega(|\cA|')$.
Conversely, since $A\subset \Int(|\cA|)$ and $A' \subset \Int(|\cA'|)$, we have
\[
A\cap A' \subset \Int(|\cA|) \cap \Int(|\cA'|) = \Int(|\cA|\cap|\cA'|) \subset \cl\bigl(\Int(|\cA|\cap|\cA'|) \bigr) = |\cA\cap\cA'|,
\]
and therefore $A\wedge A' \subset \omega(|\cA\cap\cA'|) = \omega(|\bomega(\cA\cap\cA')|) = 
\omega(|\cA\wedge\cA'|)$ by Lemma \ref{lem:commute1}.
Combining the inclusions proves the lemma.
\end{proof}

\begin{remark}
\label{rmk:other2}
Note that the evaluation map $|\cdot|\colon\sInvset^+(\cX,\cF)\to\sANbhdR(X,f)$ can be restricted to $\sAtt(\cX,\cF)$,
since every attractor is forward invariant. However, $\sAtt(\cX,\cF)$ is not a sublattice of $\sInvset^+(\cX,\cF)$,
since the lattice operations are different; $\sInvset^+(\cX,\cF)$ is a lattice under union and intersection, but
the $\wedge$ operation for $\sAtt(\cX,\cF)$ is $\cA\wedge\cA'=\bomega(\cA\cap\cA')$.
In particular
$$
|\cA\wedge\cA'|=|\bomega(\cA\cap\cA')|\subset|\cA\cap\cA'|=|\cA|\wedge|\cA'|
$$
but $|\cA\wedge\cA'|$ need not be equal to $|\cA|\wedge|\cA'|$ in general.
Therefore we cannot replace $\sInvset^+(\cX,\cF)$ by $\sAtt(\cX,\cF)$ in Diagram~$(\ref{diag:AR2})$.
The relationship between $\sAtt(\cX,\cF)$ and $\IS^+(\cX,\cF)$ is shown in Diagram (\ref{diag:comm-liftfull}).
\end{remark}

\subsection{Approximating dynamical systems}
\label{subsec:approximationDS}

In this section we address the question of how to approximate a general dynamical system $\varphi\colon \T^+\times X \to X$. 
If the time parameter $\T =\Z$, then the dynamical system is generated by the continuous map 
\[
f(\cdot) := \varphi(1,\cdot)\colon X\to X.
\]
In this case the dynamical system is represented  by a (weak) outer approximation (Definition~\ref{defn:outerapprox})
and the results of Section~\ref{subsec:at-rep-neigh} apply.
Thus, for the remainder of this section we assume that $\T =\R$, for which the question of choosing an appropriate
representation is more subtle.
For the definition of alpha and omega limit sets, and attractors and repellers we refer to \cite{KMV-1a}.

Recall that  a set $U\subset X$ is an {\em attracting neighborhood} for $\varphi\colon \R^+\times X\to X$ if
$\omega(U,\varphi) \subset \Int(U)$. A \emph{trapping region}  is a  forward invariant attracting neighborhood. 
 A set $A\subset X$ is an \emph{attractor}  if there exists an attracting neighborhood $U$  such that
$A=\Inv (U,\varphi)$ in which case $A = \omega(U,\varphi)$.
Repelling regions/neighborhoods and repellers can be define analogously, cf.\ \cite{KMV-1a}.
The notion for attracting and repelling neighborhoods is $\sANbhd(X,\varphi)$ and $\sRNbhd(X,\varphi)$, cf.\ \cite{KMV-1a}.
 % to the definition for continuous maps and we refer to \cite{KMV-1a} for a precise definition.

\begin{remark}
\label{rmk:isol}
Attractors and  repellers are examples of isolated invariant sets. In general, an  invariant set $S\subset X$ 
is an \emph{isolated invariant set} if there exist a neighborhood $U\subset X$ such that $
S = \Inv(U,\varphi) \subset \Int(U)$. The latter is called an i\emph{solating neighborhood}.
The notion of isolated invariant set is also of importance beyond attractors and repellers.
\end{remark}

The following lemma provides a relationship between  trapping regions for time-$\tau$ maps and  attracting neighborhoods for $\varphi$. 

\begin{lemma}
\label{lem:att-aux1}
If $U$ is a trapping region for the time-$\tau$ map $f(\cdot)=\varphi(\tau,\cdot)$, then $U$ is an attracting neighborhood
for $\varphi$.
\end{lemma}

\begin{proof}
Let $U$ be a trapping region for $f$ and let $A = \omega(U,f)$ denote the associated attractor.   
Set $U^\tau = \varphi([0,\tau],U)$. 
The first step of the proof is to show that $U^\tau$ is forward invariant under $\varphi$.
Observe that
\begin{equation*}
\begin{aligned}
\varphi([n\tau,(n+1)\tau],U) &= \varphi(n\tau,\varphi([0,\tau],U))\\
&= \varphi([0,\tau],\varphi(n\tau,U)) \\
&=  \varphi([0,\tau],f^n(U)) \subset  \varphi([0,\tau],U) = U^\tau,
\end{aligned}
\end{equation*}
where the inclusion follows from the forward invariance of $U$ under $f$.
Thus $\varphi([0,\infty),U^\tau) = \varphi([1,\infty),\varphi([0,\tau],U))\subset U^\tau$.

Since $U^\tau$ is forward invariant and $U^\tau = \varphi([0,\tau],U)$,
\[
\omega(U,\varphi) = \omega(U^\tau,\varphi) \subset U^\tau.
\]

Observe that 
\[
\cl\left(\bigcup_{k\geq n} f^k(U)\right)\subset \cl(\varphi([n,\infty),U)).
\]
Thus
\[
A = \bigcap_{n\ge 0} \cl\bigl( \bigcup_{k\ge n} f^k(U)\bigr)  \subset 
\bigcap_{n\ge 0}\cl(\varphi([n,\infty),U)) = \omega(U,\varphi).
\]

Since $A$ is the maximal isolated invariant set for the time $\tau$-map $f$ in $U$, it follows from \cite[Theorem 1]{Mrozek1} that $A$ is the maximal isolated invariant set for $\varphi$ in $U$. In particular,
\[
\omega(U,\varphi)= A \subset \Int(U)
\]
and hence $U$ is an attracting neighborhood. 
\end{proof}

\begin{proposition}
\label{prop:discr-att-to-varphi3}
Let $\cF\colon \cX \mvmap \cX$ be an weak outer approximation for a time-$\tau$ mapping $f=\varphi(\tau,\cdot)$. Then,
\begin{enumerate}
\item[(i)] if $\cU\subset \cX$ is a forward invariant set for $\cF$, then $U=|\cU|$ is an attracting neighborhood for  $\varphi$, and 
%, i.e. $|\cU| \in \sANbhd(X,\varphi)$;
\item[(ii)] if $\cU\subset \cX$ is a backward invariant set for $\cF$, then $U=|\cU|$ is a repelling neighborhood for  $\varphi$.
%, i.e. $|\cU| \in \sRNbhd(X,\varphi)$.
\end{enumerate}
\end{proposition}

\begin{proof}
Since $|\cU|$ is a trapping region for $f$ by Proposition~\ref{prop:discr-att-to-varphi1},  Lemma~\ref{lem:att-aux1} imply that $\omega(|\cU|,\varphi) \subset \Int|\cU|$,  which proves that $U= |\cU|$ is an attracting
neighborhood for $\varphi$.

If $\cU\in \IS^-(\cX,\cF)$, then $\cU^c\in \IS^+(\cX,\cF)$, and thus $|\cU^c| \in \sANbhd(X,\varphi)$. By  \cite[Corollary 3.26]{KMV-1a} we have that $\cl \left(\supp{\cU^c}^c\right)\in \sRNbhd(X,\varphi)$. Therefore,
\[
 |\cU^c|^\rc =\cl |\cU^c|^c = |\cU^{cc}| = |\cU| \in \sRNbhd(X,\varphi),
\]
which proves the second statement.
\end{proof}

For attracting and repelling sets we can prove a similar statement, if we consider strong outer approximations instead of 
weak outer approximations.

\begin{proposition}
\label{prop:discr-att-to-varphi4}
Let $\cF\colon \cX \mvmap \cX$ be a  outer approximation for a time-$\tau$ mapping $f=\varphi(\tau,\cdot)$. Then,
\begin{enumerate}
\item[(i)] if $\cU\subset \cU$ is an attracting set for $\cF$, then $U=|\cU|$ is an attracting neighborhood for  $\varphi$, and % i.e. $|\cU| \in \sANbhd(X,\varphi)$;
\item[(ii)] if $\cU\subset \cU$ is a repelling set for $\cF$, then $U=|\cU|$ is a repelling neighborhood for  $\varphi$. %, i.e. $|\cU| \in \sRNbhd(X,\varphi)$.
\end{enumerate}
\end{proposition}

\begin{proof}
If $\cU\in \sASet(\cX,\cF)$, then from the proof of Proposition \ref{prop:discr-att-to-varphi2} it follows that
$\varphi(k\tau,\cl |\cU|) \subset \Int |\cU|$. From the proof of Lemma \ref{lem:att-aux1} we then derive that $|\cU|$ is
an attracting neighborhood for $\varphi$.

Proving that the same holds for $\cU\in \sRSet(\cX,\cF)$ follows in the same way as in the proof of Proposition 
\ref{prop:discr-att-to-varphi3}.
\end{proof}

From the above results we conclude that
\[
\sAtt(X,\varphi) = \sAtt(X,f)\quad\hbox{and}\quad \sRep(X,\varphi) = \sRep(X,f).
\]

The fact that $\sRep(X,\varphi) = \sRep(X,f)$ as sets also implies that they are the same as lattices, since the binary operations are $\cap$ and $\cup$. This implies that $\id\colon \sRep(X,\varphi) \to \sRep(X,f)$ is a lattice isomorphism.
For attractors, we have the same result.
\begin{corollary}
\label{for:sameA}
The identity $\id\colon \sAtt(X,\varphi) \to \sAtt(X,f)$ is a lattice isomorphism.
\end{corollary}

\begin{proof}
Since  as posets $\bigl(\sAtt(X,\varphi),\subset\bigr) = \bigl(\sAtt(X,f),\subset\bigr)$, it follows that $A \wedge A'$ in $\sAtt(X,\varphi)$ is the same as $A\wedge A'$ in $\sAtt(X,f)$.
%The same follows by using the $\omega$-definition of $\wedge$.
\end{proof}

%%%%
Since the evaluation mapping $\cU \mapsto |\cU|$ yields regular closed sets and is a Boolean homomorphism,  we can summarize
the above propositions in the following commuting diagram.

\begin{equation}
\label{diag:AR2a}
\begin{diagram}
\node{}\node{\IS^+(\cX,\cF)} \arrow{e,l,<>}{^c} \arrow{s,l,V}{\imath} \arrow{sw,l,A}{\bomega}\node{\IS^-(\cX,\cF)} \arrow{s,r,V}{\imath}\arrow{se,l,A}{\balpha}	\\
\node{\sAtt(\cX,\cF)}\arrow{sse,r}{\omega(|\cdot|)}\node{\sASet(\cX,\cF)}\arrow{s,l,V}{|\cdot|}\arrow{w,l,A}{\bomega}\arrow{e,l,<>}{^c} 
\node{\sRSet(\cX,\cF)}\arrow{s,l,V}{|\cdot|}\arrow{e,l,A}{\balpha}\node{\sRep(\cX,\cF)} \arrow{ssw,r}{\alpha(|\cdot|)}	\\
\node{}\node{\sANbhdR(X,\varphi)}\arrow{s,l,A}{\omega}\arrow{e,l,<>}{^\rc} \node{\sRNbhdR(X,\varphi)}\arrow{s,l,A}{\alpha} \\
\node{}\node{\sAtt(X,\varphi)} \arrow{e,l,<>}{^*} \node{\sRep(X,\varphi)} 
\end{diagram}
\end{equation}

\section{Convergence and realization of algebraic structures via multivalued maps}
\label{sec:str-conv}
Convergent sequences of  outer approximations can be constructed as indicated in \cite{KMV0}.
For our purposes, we will use a modified notion of a convergent sequence of outer approximations and establish 
that arbitrary finite attractor lattices can be realized along such a sequence. We start with recalling some facts about outer approximations from \cite{KMV0}.

\subsection{Convergent sequences of  outer approximations}
\label{subsec:conv-fam}
Given a continuous map $f\colon X\to X$ and a grid indexed by $\cX$, \cite[Proposition 2.5]{KMV0} implies that there is a natural choice of   outer approximation which is minimal, namely
\[
\cFm(\xi) := \{ \eta\in \cX~|~ f(|\xi|) \cap |\eta| \not =\varnothing\}.
\]
We refer to this as the \emph{minimal multivalued mapping} for $f$ with respect to $\cX$.

A multivalued mapping $\cF$ \emph{encloses}  a multivalued mapping $\cF'$ if $\cF'(\xi) \subset \cF(\xi)$ for all $\xi\in \cX$.
Observe that this defines a partial order on multivalued mappings, which we  denote by $\cF'\le \cF$.

\begin{lemma}
\label{lem:enclose10}
(\cite[Corollary 2.6]{KMV0})
A multivalued mapping $\cF\colon\cX\mvmap \cX$ is a  outer approximation for $f\colon X\to X$
if and only if $\cF$ encloses
the minimal multivalued mapping $\cFm$.
\end{lemma}

Outer approximations are naturally generated by numerical approximations of $f$.
For $U\subset X$  let $\cov_{\cX}(U):=\{\eta\in\cX~|~U\cap |\eta|\neq\varnothing\}$. In this notation $\cFm(\xi)=\cov_{\cX}(f(|\xi|))$.
More generally, let ${\bm{\varrho}}\colon \cX \to [0,\infty)$ then by Lemma~\ref{lem:enclose10}
\[
\cF_{\bm\varrho}(\xi) := \bigl\{ \eta\in \cX~|~ B_{\rho(\xi)}\bigl(f(|\xi|)\bigr) \cap |\eta| \not =\varnothing\bigr\}=\textstyle{\cov_{\cX}}(B_{\rho(\xi)}(f(|\xi|)))
\]
is an outer enclosure. 
In the case ${\bm{\varrho}}(|\xi|)=\rho$ is constant for all $\xi\in\cX$, we call $\cF_\rho$ the \emph{$\rho$-minimal multivalued mapping} for $f$.
Observe that with this multivalued map the 'errors' in the image of $f$ are always smaller than $\rho$ plus the grid size expressed by
\begin{equation}
\label{eqn:diam}
\diam(\cX) = \max \{\diam(|\xi|)~|~\forall~\xi\in \cX\}.
\end{equation}

To define  convergent sequences of outer approximations we make use of  multivalued mappings $\cF$ that satisfy the {\em squeezing} condition
\begin{equation}
\label{eqn:squeeze1}
\cFm \le \cF\le \cFme.
\end{equation}
Outer approximations can always be enclosed by some $\rho$-multivalued mapping. 
Indeed, by choosing $\rho = \diam(X)$, we have that $\cFme$ encloses every 
outer approximation $\cF$. 

\begin{definition}
\label{defn:conv10}
Let $\cF_n\colon \cX_n\mvmap\cX_n$ be a sequence of outer approximations for $f\colon X\to X$. Then 
$\cF_n$ \emph{converges} if $\diam(\cX_n)\to 0$ and if there exist $\rho_n$-minimal maps $\cFmen$ with $\rho_n\to 0$ such that
\[
\cFmn \le \cF_n\le \cFmen\;\hbox{on $\cX_n$.}
\]
\end{definition}

The following proposition extends the convergence result for minimal multivalued mappings of \cite[Proposition~5.4]{KMV0} 
to $\rho$-minimal multivalued mappings.
\begin{proposition}
\label{prop:conv2}
Let $\epsilon>0$ and $k>0$. There exists $\delta >0$,
such that for every grid indexed by $\cX$, with $\diam(\cX)<\delta$, and every $\rho$-minimal mapping with
$\rho<\delta$ we have
\begin{equation}
\label{eqn:enclose3}
|\cFme^k(\xi)| \subset B_\epsilon (f^k(|\xi|)) \quad\hbox{and}\quad |\cFme^{-k}(\xi)| \subset B_\epsilon (f^{-k}(|\xi|))
\end{equation}
for all $\xi\in \cX$.
\end{proposition}

\begin{proof} First consider forward dynamics.
Since $X$ is a compact metric space,
 $f\colon X\to X$ is uniformly continuous so that for every $e>0$ there exists a $d>0$ such that
\begin{equation}
\label{eqn:uniform1}
f(B_{d}(U)) \subset B_e(f(U))
\end{equation}
for every $U\subset X$. 
Let $\delta_0 = \epsilon$.
From Equation (\ref{eqn:uniform1}) choose $\delta_i>0$ for $i=1,\cdots, k-1$ inductively so that
given $\delta_{i-1}$
\begin{equation}
\label{eqn:uniform2}
f\bigl(B_{\delta_i}(f^{k-i}(|\xi|)) \bigr)\subset B_{\delta_{i-1}/3}(f^{k-i+1}(|\xi|))
\end{equation}
for all $\xi\in \cX$.
Define $\delta  = \min_{0\le i\le k-1}\{\delta_i/3\}$. 

Let $\cX$ be the indexing set for  a grid on $X$ with $\diam(\cX) <\delta$, and consider $\cF_\rho$ on $\cX$ with $\rho<\delta$.
For the evaluation $|\cFme(\xi)|$ we have
\begin{equation}
\label{eqn:uniform3}
|\cFme(\xi)| \subset B_{\delta+\rho}(f(|\xi|)) \subset  B_{2\delta}(f(|\xi|))\subset
 B_{\delta_{k-1}}(f(|\xi|)).
 \end{equation}
We now proceed inductively. Recall that $\cFme^2(\xi) = \bigcup_{\eta\in \cFme(\xi)} \cFme(\eta)$, and for $\eta\in \cFme(\xi)$ Equation (\ref{eqn:uniform3}) implies
$\eta\subset B_{\delta_{k-1}}(f(|\xi|))$. Combining this with (\ref{eqn:uniform2}) we obtain
\[
f(|\eta|) \subset f\bigl( B_{\delta_{k-1}}(f(|\xi|)) \bigr) \subset B_{\delta_{k-2}/3}(f^2(|\xi|)),
\]
and by (\ref{eqn:uniform3})      applied to $\eta$ we have        
\[
|\cFme(\eta)| \subset B_{2\delta}(f(|\eta|)).
\]
Therefore, 
\[
|\cFme^2(\xi)| = \bigcup_{\eta\in \cFme(\xi)} |\cFme(\eta)| \subset B_{\delta_{k-2}/3+2\delta}(f^2(|\xi|))\subset B_{\delta_{k-2}}(f^2(|\xi|)).
\]

In the general case $i\ge 2$, 
for $\eta\in \cFme^{i-1}(\xi)$ we have $\eta\subset B_{\delta_{k-i+1}}(f^{i-1}(|\xi|))$ and $f(\eta) \subset 
B_{\delta_{k-i}/3}(f^i(|\xi|))$. This yields
\[
|\cFme^i(\xi)| = \bigcup_{\eta\in \cFme^{i-1}(\xi)} |\cFme(\eta)| \subset B_{\delta_{k-i}/3+2\delta}(f^i(|\xi|))\subset B_{\delta_{k-i}}(f^i(|\xi|)).
\]
After $k$ steps, we obtain $|\cFme^k(\xi)|\subset B_{\delta_{0}}(f^k(|\xi|))= B_{\epsilon}(f^k(|\xi|))$, which completes the proof for forward dynamics.

In the case of backward dynamics, we use the fact that for every $e>0$ there exists a $d>0$ such that
\begin{equation}
\label{eqn:uniform1back}
f^{-1}(B_{d}(U)) \subset B_e(f^{-1}(U))
\end{equation}
for every $U\subset X$, by continuity of $f$ and compactness of $X$.
%This implies that $\delta>0$ can be chosen so that for all $\rho<\delta$,
%$f^{-1}(B_\rho(U))\subset B_\epsilon(f^{-1}(U)).$ 
Moreover, we will use the following characterization of $\cF_\rho^{-1}$
\begin{eqnarray*}
\cF_\rho^{-1}(\xi)&=&\{\eta\in\cX~|~B_\rho(f(\eta))\cap|\xi|\neq\varnothing\}=\{\eta\in\cX~|~B_\rho(|\xi|)\cap |\eta|\neq\varnothing\}\\
&=&\{\eta\in\cX~|~|\eta|\cap f^{-1}(B_\rho(|\xi|))\neq\varnothing\}
=\textstyle{\cov_{\cX}}(f^{-1}(B_\rho(|\xi|))).
\end{eqnarray*}
%we have that 
%$$
%|\cF_\rho^{-1}(\xi)|\subset B_{\epsilon+\rho}(f^{-1}(\xi))\},
%$$
% and the rest of the arguments are analogous.

Now we proceed similarly to the previous case. Let $\delta_0 = \epsilon$.
From Equation (\ref{eqn:uniform1back}) choose $\delta_i>0$ for $i=1,\cdots, k$ inductively so that
given $\delta_{i-1}$
\begin{equation}\label{eqn:uniform2back}
f^{-1}\bigl(B_{\delta_i}(f^{-k+i}(|\xi|)) \bigr)\subset B_{\delta_{i-1}/3}(f^{-k+i-1}(|\xi|))
\end{equation}
for all $\xi\in \cX$.
Define $\delta  = \min_{0\le i\le k}\{\delta_i/3\}$. Then, since $\rho<\delta<\delta_i$,
\begin{eqnarray}
|\cF_\rho^{-1}(\xi)| &=& \bigl|\{\eta\in\cX~|~|\eta|\cap f^{-1}(B_\rho(|\xi|))\neq\varnothing\}\bigr|\nonumber \\
&\subset& \bigl|\{\eta\in\cX~|~|\eta|\cap f^{-1}(B_{\delta_{k-i+2}}(|\xi|))\neq\varnothing\}\bigr|\nonumber \\
&\subset& B_{\delta_{k-i+1}/3+\delta}(f^{-1}(|\xi|))%\subset  B_{\delta_{k-i+1}}(f^{-1}(\xi)),
\label{eqn:uniform3back}
 \end{eqnarray}
for all $i=2,\ldots,k$ and $\xi\in\cX$.

We now proceed inductively. Recall that $\cFme^{-2}(\xi) = \bigcup_{\eta\in \cFme^{-1}(\xi)} \cFme^{-1}(\eta)$, and for $\eta\in \cFme^{-1}(\xi)$ Equation (\ref{eqn:uniform3back}) implies
$\eta\subset B_{\delta_{k-1}/3+\delta}(f^{-1}(|\xi|))$. Combining this with (\ref{eqn:uniform2back}) we obtain
\begin{eqnarray*}
f^{-1}(B_\rho(|\eta|)) &\subset& f^{-1}\bigl( B_{\delta_{k-1}/3+2\delta}(f^{-1}(|\xi|)) \bigr) \\
&\subset& f^{-1}\bigl( B_{\delta_{k-1}}(f^{-1}(|\xi|)) \bigr) \subset B_{\delta_{k-2}/3}(f^{-2}(|\xi|)).
\end{eqnarray*}
%and by (\ref{eqn:uniform3})      applied to $\eta$ we have        
%\[
%|\cFme(\eta)| \subset B_{2\delta}(f(\eta)).
%\]
Therefore, 
\[
|\cFme^{-2}(\xi)| = \bigcup_{\eta\in \cFme^{-1}(\xi)} |\cFme^{-1}(\eta)| \subset B_{\delta_{k-2}/3+\delta}(f^{-2}(|\xi|)).
%\subset B_{\delta_{k-2}}(f^{-2}(\xi)).
\]

In the general case $i\ge 2$, 
for $\eta\in \cFme^{-i+1}(\xi)$ we have $\eta\subset B_{\delta_{k-i+1}/3+\delta}(f^{-i+1}(|\xi|))$ and $f^{-1}(B_\rho(|\eta|)) \subset 
B_{\delta_{k-i}/3}(f^{-i}(|\xi|))$. This yields
\[
|\cFme^{-i}(\xi)| = \bigcup_{\eta\in \cFme^{-i+1}(\xi)} |\cFme^{-1}(\eta)| \subset B_{\delta_{k-i}/3+\delta}(f^{-i}(|\xi|)).
%\subset B_{\delta_{k-i}}(f^{-i}(\xi)).
\]
After $k$ steps, we obtain $|\cFme^{-k}(\xi)|\subset B_{\delta_{0}/3+\delta}(f^{-k}(|\xi|))\subset B_{\delta_0}(f^{-k}(|\xi|))=B_{\epsilon}(f^{-k}(|\xi|))$, which completes the proof for backward dynamics.
\end{proof}

\begin{proposition}
\label{prop:conv11}
Let $\cF_n\colon\cX_n\mvmap\cX_n$ be a convergent sequence of outer approximations for $f\colon X\to X$.
For every $\epsilon>0$ and every $k>0$, there exists $N>0$ such that
\begin{equation*}
|\cF_n^k(\xi)| \subset B_\epsilon (f^k(|\xi|)) \quad\hbox{and}\quad |\cF_n^{-k}(\xi)| \subset B_\epsilon (f^{-k}(|\xi|)) 
\end{equation*}
for all $n\ge N$ and for all $\xi\in \cX_n$.
\end{proposition}

\begin{proof}
We start with the observation that $\cF_n^k \le\cFmen^k$. Indeed, suppose true for $k-1$, then
\[
\cF_n^k(\xi) = \bigcup_{\eta\in \cF_n^{k-1}(\xi)} \cF_n(\eta) \subset \bigcup_{\eta\in \cF_n^{k-1}(\xi)}\cFmen(\eta)
\subset \bigcup_{\eta\in \cFmen^{k-1}(\xi)} \cFmen(\eta) = \cFmen^k(\xi).
\]
To complete the proof we choose $\delta>0$ such that the conclusion of Proposition \ref{prop:conv2} holds.
Choose $N>0$ such that $\rho_n<\delta$ for all $n\ge N_1$, and choose $N_2>0$ such that
$\diam(\cX_n) \le \delta$ for all $n\ge N_2$. Choosing $N= \max\{N_1,N_2\}$
completes the proof.
\end{proof}

%\begin{remark}
%\label{rmk:conv12}
%The convergence used here is a slight variation on the definition used in \cite{KMV0}. One can prove that the definitions of
%convergent sequences of outer approximations are equivalent.
%In \cite{KMV0} a sequence of outer approximation $\cF_n$ on grids $\cX_n$ is said to be convergent if for every $\epsilon>0$ there exists
%an $N>0$ such that 
%\[
%|\cF_n(\xi)| \subset B_\epsilon(f(\xi))
%\]
%for all $n\ge N$ and all $\xi \in \cX_n$. 
%From the latter one can construct $\rho_n$-minimal maps $\cFmen$ with the property that $\cF_n\le \cFmen$.
%%One can repeat the proof of Proposition \ref{prop:conv2} using the above condition. 
%As a consequence
%$\diam(\cX_n)\to 0$ as $n\to \infty$. The latter follows from $|\cFm(\xi)|\subset |\cF_n(\xi)| \subset B_\epsilon(f(\xi))$.
%\end{remark}

\subsection{Realization of attractors and repellers}
\label{subsec:resol}

Theorem~\ref{thm:commutative-diagram} guarantees that forward invariant sets and attractors for an outer approximation $\cF$ of $f$
yield attracting neighborhoods  for $f$.
The converse statement is that  every attractor of a dynamical system can be realized  by an outer approximation provided the diameter of the grid is sufficiently small.  
Our goal is the stronger result that  the lattice structure of attractors can be realized.
We start by generalizing \cite[Proposition~5.5]{KMV0} from the context of minimal multivalued maps to the setting of convergent sequences of outer approximations.

\begin{proposition}
\label{prop:conv-latt-3}
Let $\cF_n\colon \cX_n \mvmap \cX_n$ be a convergent sequence of outer approximations for $f$, and let $U\in \sANbhd(X,f)$ or $U\in \sRNbhd(X,f)$. %, with $A = \omega(U)$. 
Then there exists 
$N>0$ such that for all $n\ge N$ the set  $\cU = \cov_{\cX_n}(U)$ is an attracting or repelling set  for $\cF_n$, respectively.
\end{proposition}

\begin{proof}
We consider the case $U\in \sANbhd(X,f)$, the other case is analogous.
Let $A = \omega(U)$ and let $0<d < \frac{1}{2} \dist(U,A^*)$.
Since $U$ is an attracting neighborhood by \cite[Proposition 3.21]{KMV-1a}, there exists $K>0$ such that $f^k(B_d(U))
\subset \Int(U)$ for all $k\ge K$.
This implies that for $K \le k\le 2K$ there exists an $\epsilon>0$ such that
$B_\epsilon\bigl(f^k(B_d(U))\bigr)
\subset \Int(U)$.
By Proposition \ref{prop:conv11} we can choose $N$ such that $|\cF_n^k(|\xi|)| \subset B_\epsilon(f^k(|\xi|))$ for all
$K \le k \le 2K, \xi \in \cX_n,$ and $n\ge N$. We also choose $N$ such that $\cU=\cov_{\cX_n}(U) \subset
B_d(U)$. This yields
\[
|\cF_n^k(\cU)|\subset \bigcup_{\xi\in \cU} B_\epsilon(f^k(|\xi|))\subset \Int(U) \subset |\cU|,
\]
which implies that $\cF_n^k(\cU)\subset \cU$ for all $K\le k\le 2K$. Thus 
$\cF_n^k(\cU)\subset \cU$ for all $k \ge K$, since, for example, $\cF_n^{2K+k}=(\cF_n^K)^{K+k}$ for all
$0<k\le K$. Using Proposition \ref{prop:char-attset}, this proves that $\cU$ is an attracting set when 
$n$ is sufficiently large, i.e. $\diam(\cX_n)$ is sufficiently small.
\end{proof}

This leads to the following corollary, which in the case of the minimal multivalued map
 is also a consequence of  \cite[Proposition~5.5]{KMV0}.

\begin{corollary}
\label{cor:resolution1}
Let $\cF_n\colon \cX_n\mvmap \cX_n$ be convergent sequence of outer approximations for $f$, and let $A\in \sAtt(X,f)$ be
an attractor for $f$. For every $0< d < \frac{1}{2} \dist(A,A^*)$ there exists an $N>0$ such that for every
$n\ge N$ there is an attractor $\cA_n \in \sAtt(\cX_n,\cF_n)$ and a repeller $\cR_n \in \sRep(\cX_n,\cF_n)$ with
\[
A = \omega(|\cA_n|)\subset |\cA_n| \subset B_d(A)\quad\hbox{and}\quad A^* = \alpha(|\cR_n|)\subset |\cR_n| \subset B_d(A^*).
\]
%whose dual repeller $\cA_n^*\in \sRep(\cX_n,\cF_n)$ satisfies % is the unique repeller such that 
%\[
%A^* = \alpha(|\cA^*_n|)\subset |\cA^*_n| \subset B_d(A^*).
%\]
\end{corollary}

\begin{proof}
Fix $0< d < \frac{1}{2} \dist(A,A^*)$. By Proposition~\ref{prop:conv-latt-3}, there exists $N>0$ such that $\cU_n = \cov_{\cX_n}(B_{d/2}(A))$ is an attracting set for $\cF_n$ and $\cV_n = \cov_{\cX_n}(B_{d/2}(A^*))$ 
is a repelling set for $\cF_n$
for all $n\ge N$, since $B_{d/2}(A)$ and $B_{d}(A)$ are attracting neighborhoods, and 
$B_{d/2}(A^*)$ and $B_{d}(A^*)$ are repelling neighborhoods.
Choosing $N$ large enough so that $\diam(\cX_n)<d/2$ as well implies that
$|\cU_n|\subset B_d(A)$ and $|\cV_n|\subset B_d(A^*)$. 
Moreover, if $\cA_n=\bomega(\cU_n)$ and $\cR_n=\balpha(\cV_n)$, then 
$\cA_n\subset \cU_n$ and $\cR_n\subset \cV_n$ implies $|\cA_n|\subset B_d(A)$ and $|\cR_n|\subset B_d(A^*)$.
By Proposition~\ref{lem:commute1},   $\omega(|\cA_n|)=\omega(|\bomega(\cU_n)|)=\omega(|\cU_n|)$. 
Moreover $A=\omega(B_{d/2}(A))\subset\omega(|\cU_n|)\subset\omega(B_d(|\cU_n|)=A$ so that
$A=\omega(|\cA_n|)$. By Proposition~\ref{prop:discr-att-to-varphi2}, $|\cA_n|$ is an attracting neighborhood 
so that $\omega(|\cA_n|)\subset|\cA_n|$, which completes the proof for the attractor. The same argument holds for the repeller.
\end{proof}

\subsection{Posets, Lattices and Grids}
\label{subsec:PLG}
In the previous subsection we established that any attractor or repeller in a system can be realized via a multivalued map if the diameter of the grid is sufficiently small.  Furthermore, these discrete attractors and repellers correspond to arbitrarily narrow attracting and repelling neighborhoods, respectively.
To prove that the lattice structures can be realized via multivalued maps requires more subtle constructions based on the lattice and poset structures of grids and multivalued maps.

We begin by providing a systematic means of generating convergent sequences of outer approximations.
For the sake of  simplicity we will abuse notation throughout this subsection and often refer to the grid $\setof{\supp{\xi}\mid \xi\in\cX}$  by its indexing set $\cX$.
%Throughout this section $\cX$ denotes a fixed indexing set for a fixed grid on $X$.
%For the sake of  simplicity we will abuse notation and often refer to the grid $\setof{\supp{\xi}\mid \xi\in\cX}$ as $\cX$.

%\corrc we are using $\cX$ as both grid and indexing set for grid.  KM
%I put some bars in the definitions below, RC
%<<>>

%\corrc what follows is now a full page of definitions only (4 in a row) -- OK with me, but 
%just pointing it out -- BK 9/17/14  OK with me KM 9/17/14 <<>>

\begin{definition}
\label{defn:refinement}
A grid $\cX'$ on $X$ is a {\em refinement of $\cX$}, denoted by $\cX'\le \cX$, if for every $\xi'\in \cX'$ there exists exactly one $\xi\in \cX$ such that
$|\xi'|\subset |\xi|$. 
\end{definition}

Refinement defines a partial order on the space of grids $\sGrid(X)$, which can be used to compare multivalued maps.

\begin{definition}
\label{defn:cofil1}
Let $\cX' \le \cX$ be grids on $X$ and let $\cF\colon \cX \mvmap \cX$ and $\cF\colon\cX'
\mvmap \cX'$ be multivalued mappings. A partial order on multivalued mappings and grids is given by
\[
\cF'\le \cF \quad\text{if}\quad \supp{\cF'(\cU')} \subseteq \supp{\cF(\cU)}\ \text{for all}\ \supp{\cU'} = \supp{\cU}
\]
where $\cU'\in \sSet(\cX')$ and $\cU\in \sSet(\cX)$.
\end{definition}

\begin{definition}
\label{defn:commref}
The \emph{common refinement} of $\cX$ and $\cX'$ is the grid 
\[
\{|\xi|\wedge|\xi'|~|~\hbox{$\xi\in\cX$ and $\xi'\in\cX'$ with $|\xi|\wedge|\xi'|\neq\varnothing$}\}.
\]
The set of all pairs $\xi$ and $\xi'$ for which $|\xi|\wedge |\xi'|\neq\varnothing$ is
 an indexing set for this grid. We denote this indexing set by $\cX\wedge\cX'$ and an individual index
by $\xi\wedge\xi'$. 
\end{definition}
Note that whenever the index $\xi\wedge\xi'$ is used, it is implied that $|\xi|\wedge|\xi'|\neq\varnothing.$

The \emph{common refinement of multivalued mappings}  $\cF\colon \cX \mvmap \cX$ and $\cF\colon\cX'
\mvmap \cX'$ is given by
\begin{equation}
\label{eqn:common-milt}
(\cF\wedge \cF')(\xi\wedge\xi') := \{ \eta\wedge\eta' ~|~\eta\in \cF(\xi),~\eta'\in \cF'(\xi)\}.
\end{equation} 
Observe that $\cF\wedge \cF'\colon \cX\wedge \cX'\mvmap \cX\wedge \cX'$.

\begin{definition}
\label{defn:cofilt}
A \emph{cofiltration of grids} is a sequence $\{\cX_n\}_{n\in \N_0} \subset \sGrid(X)$ of refinements so that
\[
\cX_0 \ge \cX_1 \ge \cdots \ge \cX_n\ge \cdots
\]
Furthermore, given a  cofiltration of grids $\{\cX_n\}_{n\in \N_0}$, a sequence of multivalued mappings $\cF_n\colon \cX_n \mvmap \cX_n$, which satisfies
\[
\cF_0 \ge \cF_1 \ge \cdots \ge \cF_n \ge \cdots
\]
is called a \emph{cofiltration of multivalued mappings}.
The function $\diam\colon \sGrid(X) \to \R^+$ is order-preserving so that $\diam(\cX') \le \diam(\cX)$ for any pair $\cX'\le \cX$.
If $\diam(\cX_n) \to 0$ as $n\to \infty$, then a cofiltration $\{\cX_n\}_{n\in \N_0}$ of grids is said to be \emph{contracting}.
\end{definition}

Given any sequence of grids $\{\cX_n\}$ for which $\diam(\cX_n) \to 0$, we can construct a contracting cofiltration as
follows
\[
\cX_0 \ge \cX_0\wedge \cX_1 \ge \cX_0\wedge \cX_1 \wedge \cX_2 \ge \cdots \ge \bigwedge_{i=0}^n \cX_i \ge \cdots .
\]
If $\cF_n\colon \cX_n \mvmap \cX_n$ is a sequence of multivalued mappings with $\diam(\cX_n) \to 0$, then
\begin{equation}
\label{eqn:cofil2}
\bigwedge_{i=1}^n \cF_i \colon \bigwedge_{i=1}^n \cX_i \mvmap \bigwedge_{i=1}^n \cX_i
\end{equation}
is a cofiltration of multivalued mappings.

%\corrl we have defined convergent and cofiltration do we need to state so formally (or at all)
%the obvious -- convergent cofiltration? -- BK 9/17/14 Agree with proposed change KM 9/18/14
%<<
%\begin{definition}
%\label{defn:convcofil}
%A sequence of multivalued mappings 
% $\cF_n\colon \cX_n \mvmap \cX_n$ is called {\em convergent cofiltration} of multivalued mappings
% if the sequence is both a cofiltration and convergent.
% \end{definition}
% ||
% >>
% 
%
%
%\corrl -- BK 9/17/14  I like the changes KM 9/18/14
%<<
%An important example of convergent cofiltrations of multivalued mappings are given by $\rho$-minimal multivalued mappings on
%contracting cofiltrations of grids.
%||
From an algorithmic point of view, given a grid $\cX_n$, one designs an algorithm to construct $\cF_n:\cX_n\mvmap\cX_n$.
The monotonicity of images of $\cF_n$ required for a cofiltration does not automatically follow
from the fact that $\cX_n$  is a cofiltration of grids, and the construction in equation~$(\ref{eqn:cofil2})$ 
is often inefficient in practical applications. From a theoretical point of view,
an important example of a convergent cofiltration of multivalued mappings is given by the $\rho$-minimal multivalued mappings on a
contracting cofiltration of grids. Theorems~\ref{thm:conv-latt-2} and~\ref{thm:conv-latt-3}
contrast what is attainable through a convergent cofiltration versus simply a convergent sequence of
multivalued mappings, see Remark~\ref{rmk:mono}.
%>>

Some properties of dynamics are preserved through cofiltrations. We only present the following which we make use of in the proof of Theorem~\ref{thm:conv-latt-2}.

\begin{proposition}
\label{prop:cofiltmap}
Let $\{\cX_n\}_{n\in \N_0}$ be a  cofiltration of grids and let $\cF_n\colon \cX_n\mvmap \cX_n$ be a  cofiltration of multivalued mappings.
Consider a collection of subsets $\cW_n\subset \cX_n$ such that $\supp{\cW_n} = \supp{\cW_m}\subset X$. 
If $m > n$ and $\cW_n \in \IS^-(\cX_n,\cF_n)$, then $\cW_m \in \IS^-(\cX_m,\cF_m)$.
\end{proposition}

\begin{proof}
We need to show that $\cF_m^{-1}(\cW_m)\subset \cW_m$. Since $\cW_n \in \IS^-(\cX_n,\cF_n)$,  it is sufficient to show that $\supp{\cF_m^{-1}(\cW_m)} \subset \supp{\cF_n^{-1}(\cW_n)}$ so that
\[
\supp{\cF_m^{-1}(\cW_m)} \subset \supp{\cF_n^{-1}(\cW_n)} \subset \supp{\cW_n} = \supp{\cW_m}.
\]

Let $\beta_n\in\cW_n$ and $\beta_m\in\cW_m$ satisfy $\supp{\beta_m} \subset \supp{\beta_n}$. 
Consider $\eta\in\cF_m^{-1}(\beta_m)$.
By definition of cofiltration, there exists $\xi\in\cX_n$ such that $\supp{\eta}\subset\supp{\xi}$. 

Now $\beta_m\in\cF_m(\eta)$ which implies that
\[
\supp{\beta_m}\subset \supp{\cF_m(\eta)}\subset\bigcup_{\supp{\zeta}\subset\supp{\xi}}|\cF_m(\zeta)|=
\supp{\cF_m\left(\bigcup_{\supp{\zeta}\subset\supp{\xi}}\zeta\right)} \subset |\cF_n(\xi)|
\]
where the last inclusion follows from the definition of cofiltration. 
Since $\cF_n(\xi)$ is a union of elements of $\cX_n$, and $\supp{\beta_m}\subset \supp{\beta_n}$, we must have $\supp{\beta_n}\subset|\cF_n(\xi)|$, which implies $\beta_n\in\cF_n(\xi)$ and equivalently $\xi\in\cF_n^{-1}(\beta_n)$. 
Hence $\xi\in\cF_n^{-1}(\cW_n)$ and $\supp{\eta}\subset\supp{\xi}\subset \supp{\cF_n^{-1}(\cW_n)}$. 
Thus, if  $\eta\in\cF_m^{-1}(\cW_m)$, then $\supp{\eta}\subset \supp{\cF_n^{-1}(\cW_n)}$, and
therefore $\supp{\cF_m^{-1}(\cW_m)} \subset \supp{\cF_n^{-1}(\cW_n)}$.
\end{proof}

The realization of the lattice structures of attractors and repellers is presented in the language of lifts which are defined as follows.

\begin{definition}
\label{defn:lift}
Let $\sL$, $\sK$, and $\sH$ be bounded distributive lattices. Let $g\colon \sL \rightarrowtail \sK$ be a lattice monomorphism and $h\colon \sH\twoheadrightarrow \sK$ be a lattice
epimorphism. A lattice homomorphism $\ell\colon \sL\to \sH$ is a {\em lift of $g$ through $h$} if $g = h\circ \ell$.
\end{definition}
Observe that a lift is necessarily a lattice monomorphism.

As is made clear in the next section our goal is to construct a lift. 
To do this we make use of concepts from the theory of distributive lattices.
Recall that an element $c\in\sL$ is 
{\em join-irreducible} if
\begin{enumerate}
\item[(a)] $c\neq 0$ and
\item[(b)] $c=a\vee b$ implies $c=a$ or $c=b$ for all $a,b\in\sL$.
\end{enumerate}
The set of join-irreducible elements in $\sL$ is denoted by $\sJ(\sL)$.
Note that $\sJ(\sL)$  is a poset as a subset of $\sL$.
Observe that $c$ is join-irreducible if and only if there exists a unique element $a\in\sL$ satisfying
$a<c$ and there does not exist $b\in\sL$ such that $a<b<c$.
The element $a\in \sL$ is called the \emph{immediate predecessor} of $c$ and denoted by
\begin{equation}
\label{eqn:immpred}
a = \pred c.
\end{equation}

Given a finite poset $\sP$ with partial order $\le$, then the down-set of $p\in \sP$ is given by
$\down p=\{q\in\sP~|~q\le p\}$. These sets generate a finite distributive lattice
$\sO(\sP)$ in $\sSet(\sP)$ called the lattice of {\em down-sets}.
The elements $\down p$ are the join-irreducible elements in $\sO(\sP)$.
Birkhoff's representation theorem for finite distributive lattices $\sL$ 
states that $\sL \cong \sO(\sJ(\sL))$ cf.\ \cite{Davey:2002p5923}.

%-----
%
%Given a finite poset $\sP$  with partial order $\le$ the {\em down-set}  of $p\in \sP$ is given by
%$\down p=\{q\in\sP~|~q\le p\}$.
%The collection of all down-sets  of a finite poset $\sP$ generates a finite distributive lattice denoted by $\sO(\sP)$ with respect to $\vee=\cup$ and $\wedge=\cap$. 
%The elements $\down p$ are exactly the join-irreducible elements in $\sO(\sP)$.
%Birkhoff's representation theorem for finite distributive lattices $\sL$ states that $\sL \cong \sO(\sJ(\sL))$ cf.\ \cite{Davey:2002p5923}.
%%\corrc note about $\sP$ KM <<>>

\begin{remark}
\label{rem:birkhoff}
Birkhoff's representation theorem
 allows us to recast the definition of $\ell$ being a lift of $g$ through $h$ via the following commutative diagram
\begin{equation}
\label{diag:lift12}
\begin{diagram}
\node{~} \node{\sH} \arrow{s,r,A}{h}\\
\node{\sO(\sP)} \arrow{ne,l,..,V}{\ell} \arrow{e,l,V}{g}  \node{\sK} 
\end{diagram}
\end{equation}
for any poset $\sP$ isomorphic to $\sJ(\sL)$.
For the sake of simplicity  we will abuse notation and use $\ell\colon \sL\to \sH$ and $\ell\colon \sO(\sP)\to \sH$ to denote two distinct, but equivalent homomorphisms.
\end{remark}

%\corrs where?? - BK 9/4/14
%<<As indicated above, we||We>> 
We are interested in the case in which $\sH$ is a Boolean algebra, or $\sH$ is embedded in a Boolean algebra,
 and thus we want to extend  the lift
$\ell\colon \sL \to \sH$ to a Boolean homomorphism. To do this we make use of the Booleanization functor.  The natural extension
$\sL \hookrightarrow \sSet(\sJ(\sL))$ is called the Booleanization of $\sL$ and is denoted by $\sB(\sL) = \sSet(\sJ(\sL))$. Booleanization is a covariant functor and 
the induced homomorphism $\sB(\ell) \colon \sB(\sL) \to \sH$ is Boolean and $\sB(\ell)|_\sL = \ell$, cf.\ \cite[Definition 9.5.5]{vickers} and \cite[Corollary 20.11]{miraglia}.

%\corrc
%We need to be careful here. In the application $\sH$ is naturally embedded into a Boolean algebra and use that. It makes Booleanization
%easier but $h$ is defined from $\sH$ to $\sK$ and not from the Booleanization. This piece needs to be written with a bit more detail, RC
%<<>>

The combination of Birkhoff's representation theorem and the Booleanization functor allows one to give the following representation of  $\ell$:
\begin{equation}
\label{eqn:inatom}
\ell(\alpha) = \bigvee_{p\in \alpha} c_p
\end{equation}
where $c_p := \ell(\gamma) \setminus \ell(\beta)\in\sH$ and is independent of the choice of $\beta,\gamma\in \sO(\sJ(\sL))$
for which $\gamma \setminus \beta = p$, cf.\ Theorem 2.1 and Proposition 2.3 in \cite{KMV-1a}. 
Observe that the $c_p$ are atoms of $\sH$, i.e.\ if $p\not = p'$, then $c_p \wedge c_{p'} = 0$.

With these abstract constructions in mind, we now turn to the objects of interest. 
As is detailed in the next section, we are interested in lifts of the form $\ell\colon \sR \to \IS^-(\cX,\cF)$ where $\sR\subset \sRep(X,f)$ is a finite sublattice of repellers and  such that
$\alpha(|\ell(R)|) = R$ for all $R \in \sR$.
Since $\IS^-(\cX,\cF)$ embeds (as a lattice) into the Boolean algebra $\sSet(\cX)$, we can adopt the perspective that $\ell\colon \sR \to \sSet(\cX)$ is a lattice monomorphism.  
If we represent $\sR$ by a lattice isomorphism $\sO(\sP) \cong \sR$, where $\sP\cong\sJ(\sR)$, then
application of the Booleanization functor to $\ell\colon\sO(\sP) \to \IS^-(\cX,\cF)$ yields the Boolean monomorphism
$\sB(\ell) \colon \sSet(\sP)\to \sSet(\cX)$.  This allows us to represent $\ell$ by
\begin{equation}
\label{eqn:props-10}
\ell(\alpha) = \bigcup_{p\in\alpha} \cV_p
\end{equation}
where 
\begin{equation}
\label{eq:Vp}
\cV_p := \ell(\gamma)\setminus \ell(\beta)\subset \cX
\end{equation} 
for any choice of $\beta,\gamma\in\sO(\sP)$ such that $\setof{p}=\gamma\setminus\beta$. Since $\setof{\cV_p\mid p\in\sP}$ are atoms, $\cV_p\cap \cV_{p'} = \varnothing$ if $p\neq p'$. 

\begin{proposition}
\label{prop:Boolappl}
Let $\sO(\sP)$ be a finite distributive lattice, and let $\ell\colon \sO(\sP) \to \sSet(\cX)$ be a lattice monomorphism. 
Then $\sP$ is an indexing set for a grid on $X$ whose elements are
$ \supp{\cV_p} $
under the evaluation map $\supp{\cdot}\colon \sSet(\cX)\to \scrR(X)$ and
$\cV_p := \ell(\gamma)\setminus \ell(\beta)\subset \cX$ for any choice of $\beta,\gamma\in\sO(\sP)$ such that $\setof{p}=\gamma\setminus\beta$.
\end{proposition}

\begin{proof}
By Corollary~\ref{cor:grid->lattice} the evaluation map $|\cdot|\colon \sSet(\cX)\to \scrR(X)$ is Boolean, and thus the composition $\supp{\sB(\ell)}\colon \sSet(\sP)\to \scrR(X)$ is Boolean.  
In particular $\supp{\sB(\ell)(\sP)}$ is a finite subalgebra of $\scrR(X)$.  
Hence the atoms of $\supp{\sB(\ell)(\sP)}$, which are $\setof{\supp{\cV_p} \mid p\in \sP}$, form a grid of $X$.
\end{proof}

\begin{proposition}
Let $\sO(\sP)$ be a finite distributive lattice and let $\ell\colon \sO(\sP) \to \sSet(\cX)$ be a lattice monomorphism.  Then
\[
\supp{\cV_p} \cap \Int \supp{\ell(\alpha)} = \varnothing \quad\text{for all}\quad p\not\in\alpha \in \sO(\sP).
\]
\end{proposition}

\begin{proof}
By \eqref{eqn:props-10} and \eqref{eq:Vp} we have that $\cV_p \cap l(\alpha) = \varnothing$.
Because $\setof{\cV_p\mid p\in\sP}$ is a grid for $X$ we obtain
\[
\supp{\cV_p} \wedge \supp{\ell(\alpha)} = \varnothing,\quad \forall  p\not\in \alpha \in \sO(\sP)
\]
which is equivalent to
$\supp{\cV_p} \cap \Int \supp{\ell(\alpha)} = \varnothing$
by Lemma~\ref{lem:equivregcl}.
\end{proof}

Because $\setof{\cV_p\mid p\in \sP}$ are atoms, $\supp{\cV_p}\wedge \supp{\cV_{p'}} = 0$ under the lattice operation of $\scrR(X)$.
Since in this lattice $\wedge \neq \cap$, we cannot conclude  that $\supp{\cV_p}\cap \supp{\cV_{p'}} = \varnothing$.  
More generally, since $\supp{\sB(\ell)}\colon \sO(\sP)\to \scrR(X)$ is a lattice homomorphism
$\supp{\ell(\gamma)} \wedge \supp{\ell(\alpha)} =  \supp{\ell(\gamma\cap \alpha)}$, but this does not imply that
$\supp{\ell(\gamma)} \cap \supp{\ell(\alpha)} =  \supp{\ell(\gamma\cap \alpha)}$ since $\supp{\cdot}\colon \IS^-(\cX,\cF)\to \sRNbhd(X,f)$ is not a lattice homomorphism. It is a homomorphism if we replace  $\sRNbhd(X,f)$ by $\sRNbhdR(X,f)$. 
In order to obtain results that hold in $\sRNbhd(X,f)$,
 we introduce the following concept.

\begin{definition}
Let $\sO(\sP)$ be a finite distributive lattice. A  lattice monomorphism $\ell\colon \sO(\sP) \to \sSet(\cX)$ is {\em well-separated} if
\begin{equation}
\label{eqn:props-11}
|\cV_p| \cap |\cV_{p'}| = \varnothing\quad\hbox{for all $p~\Vert~{p'}$},\quad p,p'\in\sP
\end{equation}
where the  $p~\Vert~{p'}$ indicates that $p$ and ${p'}$ are incomparable, i.e. $p\not\le {p'}$ and $p'\not \le {p}$.
\end{definition}

\begin{proposition}
\label{prop:well-sep}
If $\sO(\sP)$ is a finite distributive lattice and  $\ell\colon \sO(\sP) \to \sSet(\cX)$ is  well-separated, then
\begin{equation}
\label{eqn:well-sep}
\supp{\ell(\gamma)} \cap \supp{\ell(\alpha)} = \supp{\ell(\gamma)} \wedge \supp{\ell(\alpha)}.
\end{equation}
\end{proposition}

\begin{proof}
Observe that
\[
\supp{\ell(\gamma)} \cap \supp{\ell(\alpha)}= \supp{ \bigcup_{p\in\gamma} \cV_p } \cap \supp{ \bigcup_{q\in\alpha} \cV_q  }
= \bigcup_{p\in\gamma\atop q\in\alpha} \left( \supp{\cV_p} \cap \supp{\cV_q} \right).
\]
By \eqref{eqn:props-11}, if $\supp{\cV_p} \cap \supp{\cV_q} \neq\varnothing$, then either $p\le q$ or $q\le p$.
Since $\gamma$ and $\alpha$ are down sets, this implies that $p\in\gamma\cap\alpha$ or
$q\in\gamma\cap\alpha$ respectively, and hence $\supp{\cV_p} \cap \supp{\cV_q} \subset \supp{\cV_r}$ for
some $r\in\gamma\cap\alpha$. 
Therefore
\begin{eqnarray*}
\supp{\ell(\gamma)} \cap \supp{\ell(\alpha)} &=& \bigcup_{p\in\gamma\atop q\in\alpha} \supp{\cV_p} \cap \supp{\cV_q}
=\bigcup_{r\in\gamma\cap\alpha} \supp{\cV_r} = \supp{ \bigcup_{r\in\gamma\cap\alpha}\cV_r }\\
&=&  \supp{\ell({\gamma\cap\alpha})} =  \supp{\ell(\gamma)\cap \ell(\alpha)}
= \supp{\ell(\gamma)} \wedge \supp{\ell(\alpha)}
\end{eqnarray*}
where the last two equalities follow from the fact that $\ell$ is a lattice homomorphism and
Corollary~\ref{prop:grid->lattice}, respectively.
\end{proof}

Let $\lambda\in\sO(\sP)$ which is a subposet of $\sP$. Note that $0\in\sO(\lambda)$. However, if $\lambda\neq\sP$, then $\sP\notin\sO(\lambda)$, and hence $\sO(\lambda)$ is not a sublattice of $\sO(\sP)$. 
Therefore we define $\lambda^\top$ to  be the poset $\lambda\cup\{\top\}$ where the 
additional top element $\top$ has relations $p\le\top$ for all $p\in\lambda$.
%{\color{blue}
Observe that as a set $\top =\sP$.
%}
Then
$$
\sO(\lambda^\top)\approx\{\alpha\in\sO(\sP)~|~\alpha\subset\lambda\;\hbox{or}\;\alpha=\sP\}
$$
making  $\sO(\lambda^\top)$  a sublattice of $\sO(\sP)$. Booleanization implies
$\sB(\sO(\lambda^\top))\subset\sB(\sO(\sP))=\sSet(\sP)$.

\begin{definition}
\label{defn:part-lift}
Let $\lambda\in \sO(\sP)$. 
A lattice homomorphism $\ell:\sO(\lambda^\top) \to \sH$ is 
a {\em partial lift of $g$ on $\sO(\lambda^\top)$} in Diagram \eqref{diag:lift12} if
\[
h(\ell(\beta)) = g(\beta)\;\;\hbox{for all $\beta \leq \lambda$.}
\] 
\end{definition}

Note that by the above definition $\ell(1) = 1$, since $\ell$ is a lattice homomorphism.

% !TEX root = ./attractors-IIs2.tex

\subsection{Realization of attractor and repeller lattices}
\label{subset:lift-grid}

In this section,  using ideas from \cite{KMV-1a}, we prove that the lattice structures can be realized via multivalued maps.

\begin{theorem}
\label{thm:conv-latt-1}
Let $f\colon X \to X$ be a continuous mapping on a compact metric space $X$.
Let $\cF_n\colon  \cX_n \mvmap \cX_n$ be a convergent cofiltration of outer approximations for $f$ defined
on a contracting cofiltration of grids on $X$. 
If $\sA \subset \sAtt(X,f)$ is a finite sublattice, then there exists an $n_\sA$ such that for all $n\ge n_\sA$ there exists a lift $\ell_n\colon \sA \to \IS^+(\cX_n,\cF_n)$ of the inclusion map $\iota\colon \sA \rightarrowtail \sAtt(X,f)$ through $\omega(\supp{\cdot})\colon \IS^+(\cX_n,\cF_n) \to \sAtt(X,f)$, i.e.\ the following diagram commutes
\[
\begin{diagram}
\node{~} \node{\IS^+(\cX_n,\cF_n)} \arrow{s,r}{\omega(|\cdot|)}\\
\node{\sA}  \arrow{e,l,V}{\imath} \arrow{ne,l,..,V}{\ell_n}  \node{\sAtt(X,f).} 
\end{diagram}
\]
Furthermore, $\ell_n$ can be chosen such that $\supp{\ell_n(\sA)}$ is a sublattice of $\sANbhd(X,f)$.
\end{theorem}

We do not know of a direct proof of Theorem~\ref{thm:conv-latt-1}. 
The difficulty arises from the fact that $\wedge = \cap$ for the lattice $\IS^+(\cX_n,\cF_n)$, but $\wedge\neq \cap$ for the lattice $\sAtt(X,f)$.
Recall, however, that $\wedge = \cap$ for the lattice $\sRep(X,f)$.
With this in mind we prove the following analogous theorem for repellers. By the proof of \cite[Theorem 1.2]{KMV-1a} and in particular
\cite[commutative diagram (24)]{KMV-1a}, the Theorem~\ref{thm:conv-latt-2} for repellers implies Theorem~\ref{thm:conv-latt-1}
for attractors by a duality argument.

\begin{theorem}
\label{thm:conv-latt-2}
Let $f\colon X \to X$ be a continuous mapping on a compact metric space $X$.
Let $\cF_n\colon  \cX_n \mvmap \cX_n$ be a convergent  cofiltration of outer approximations for $f$ defined
on a contracting cofiltration of grids on $X$. 
If $\sR \subset \sRep(X,f)$ is a finite sublattice, then there exists an $n_\sR$ such that for all $n\ge n_\sR$ there exists a lift $\ell_n\colon \sR \to \IS^-(\cX_n,\cF_n)$ of the inclusion map $\iota\colon \sR \rightarrowtail \sRep(X,f)$ through $\alpha(\supp{\cdot})\colon \IS^-(\cX_n,\cF_n) \twoheadrightarrow \sRep(X,f)$, i.e.\ the following diagram commutes
\begin{equation}
\label{eq:repellerliftdiagram}
\begin{diagram}
\node{~} \node{\IS^-(\cX_n,\cF_n)} \arrow{s,r}{\alpha(|\cdot|)}\\
\node{\sR}  \arrow{e,l,V}{\imath} \arrow{ne,l,..,V}{\ell_n}  \node{\sRep(X,f).} 
\end{diagram}
\end{equation}
Furthermore, $\ell_n$ can be chosen such that $\supp{\ell_n(\sR)}$ is a sublattice of $\sRNbhd(X,f)$.
\end{theorem}

\begin{proof}
Observe that by Remark~\ref{rem:birkhoff} to prove \eqref{eq:repellerliftdiagram} it is sufficient to prove the existence of $\cU_n$ such that the following diagram commutes
\begin{equation}
\label{eq:diagramproof}
\begin{diagram}
\node{~} \node{\IS^-(\cX_n,\cF_n)} \arrow{s,r,A}{\alpha(|\cdot|)}  \arrow{e,l,J}{}\node{\sSet(\cX_n)}\\
\node{\sO(\sP)} \arrow{ne,l,..,V}{\cU_n} \arrow{e,l,V}{R_n}  \node{\sR_n} \arrow{e,J}\node{\sRep(X,f)} 
\end{diagram}
\end{equation}
where $\sR_n := \alpha\left(\supp{\IS^-(\cX_n,\cF_n)} \right)$.
Viewing $\cU_n$ as a map into $\sSet(\cX_n)$, by Proposition~\ref{prop:Boolappl} we obtain a grid $\setof{\supp{\cV_{n,p}}\mid p\in \sP}$ for $X$.

Observe that if $\alpha, \alpha'\in \sO(\sP)$ satisfy $\alpha\cap\alpha' = \varnothing$, then 
$R_n(\alpha)\cap R_n(\alpha')=\varnothing$, and hence there exists $d_\varnothing >0$ such that 
\begin{equation}
\label{eq:sepRep}
B_{d_\varnothing}(R_n(\alpha)) \cap B_{d_\varnothing}(R_n(\alpha')) = \varnothing.
\end{equation}
Since $\sO(\sP)$ is finite, we can choose $d_\varnothing >0$ such that \eqref{eq:sepRep} is satisfied for
all $\alpha, \alpha'\in \sO(\sP)$ satisfying $\alpha\cap\alpha' = \varnothing$.
By Corollary~\ref{cor:resolution1} there exists $n_{d_\varnothing}>0$  such that for each $\alpha\in \sO(\sP)$ there is an associated  discrete repeller $\cR_n(\alpha)\in \IS^-(\cX_n,\cF_n)$ satisfying
\[
R_n(\alpha) \subset |\cR_n(\alpha)|\le B_{d_\varnothing}(R_n(\alpha))
\]
for all $n\ge n_{d_\varnothing}$.

Having established the necessary notation we provide a proof by induction making use of partial lifts.
To establish the initial induction step choose $n\ge n_{d_\varnothing}$ and let $q\in \sP$ be minimal.
Observe that $\setof{q}\in \sO(\sP)$.
Define 
\[
\cU_n(\setof{q}) := \cR_n(\setof{q}).
\]
Observe that $\cU_n\colon \sO(\setof{q}^\top)\to \IS^-(\cX_n,\cF_n)$ defines a partial lift of $R_n$ through
$\alpha(\supp{\cdot})$ that satisfies the following three conditions:
\begin{description}
\item[C1] $\cV_{n,p}\cap\cV_{n,p'} = \varnothing$ for all $p\not = {p'}$;
\item[C2] $\supp{\cV_{n,p}} \cap R_n(\alpha) = \varnothing$ if $p \not\in\alpha$;
\item[C3] $\supp{\cV_{n,p}} \cap \supp{\cV_{n,p'}} = \varnothing$ for all $p~\Vert~{p'}$,
\end{description}
where $p,p'\in \setof{q}, \alpha\in\sO(\sP),$ and $\cV_{n,p}$ is defined by \eqref{eq:Vp}.

Assume that for some $\lambda \in \sO(\sP)$ and some $n_\lambda\ge  n_{d_\varnothing}$ there exists a partial lift  $\cU_{n_\lambda}\colon \sO(\lambda^\top)\to  \IS^-(\cX_{n_\lambda},\cF_n)$ of $R_n$ through
$\alpha(\supp{\cdot})$ which satisfies Conditions {\bf C1} - {\bf C3} for $p,p'\in \lambda$ and $\alpha\in\sO(\sP).$
Furthermore, given $\cU_{n_\lambda}$ \eqref{eq:Vp} defines $\setof{\cV_{n_\lambda,p}\mid p\in \sP}$.
We now show that for  $n$  sufficiently large  a new partial lift can be constructed on a down set in $\sO(\sP)$ with one additional element.

Let $q\in \sP\setminus \lambda$ be minimal. 
Define $\mu = \down q$. 
By condition {\bf C2}, 
\begin{enumerate}
\item [(i)] if $p\in \lambda\setminus \mu$ then $\supp{\cV_{n_\lambda,p}} \cap R_{n_\lambda}(\mu) = \varnothing$.
\end{enumerate}
Since $R_{n_\lambda}$ is a lattice homomorphism, and $\cU_{n_\lambda}$ is a partial lift,
\begin{enumerate}
\item [(ii)] if $\mu\cap \alpha\subset \lambda$, i.e.\ if $q\not\in\alpha\in\sO(\sP)$, then
\[
R_{n_\lambda}(\mu) \cap R_{n_\lambda}(\alpha) = R_{n_\lambda}({\mu\cap \alpha})\subset R_{n_\lambda}(\lambda) \subset \Int |\cU_{n_\lambda}(\lambda)|.
\]
\end{enumerate}

Property (i) implies that there exists a $d_\lambda>0$ such that $\supp{\cV_{n_\lambda,p}} \cap B_{d_\lambda}(R_{n_\lambda}(\mu)) =\varnothing$
for all $p\in \lambda\setminus \mu$.
Property (ii) is equivalent to $\bigl( R_{n_\lambda}(\mu) \setminus \Int |\cU_{n_\lambda}(\lambda)|\bigr) \cap R_{n_\lambda}(\alpha)  = \varnothing$, which
implies that if $q\not\in\alpha$, then $\cl\bigl( R_{n_\lambda}(\mu)\setminus |\cU_{n_\lambda}(\lambda)|\bigr) \cap R_{n_\lambda}(\alpha) = \varnothing$.
We can therefore choose $d_\lambda$ small enough such that 
\begin{enumerate}
\item [(i)']  if $p\in \lambda\setminus \mu$, then $\supp{\cV_{n_\lambda,p}} \cap B_{d_\lambda}(R_{n_\lambda}(\mu)) = \varnothing$, and
\item [(ii)']  if $q\not\in\alpha\in\sO(\sP)$, then $\cl \bigl(B_{d_\lambda}(R_{n_\lambda}(\mu)) \setminus \supp{\cU_{n_\lambda}(\lambda)}\bigr) \cap R_{n_\lambda}(\alpha)  = \varnothing$.
\end{enumerate}
Observe that throughout this discussion we have been working with $n = n_\lambda$ and thus the fixed evaluation map $\supp{\cdot} = \supp{\cdot}_{n_\lambda} \colon \cX_{n_\lambda} \to \scrR(X)$.
Now we must change $n$, and thus the evaluation map also changes. Whenever the choice of evaluation map is clear, we continue to denote it by $\supp{\cdot}$.

By Corollary~\ref{cor:resolution1}  we can choose $n_{d_\lambda} \ge n_\lambda$ such that for any $n\ge n_{d_\lambda}$ the repeller $\cR_{n}(\mu)$ guaranteed by this corollary  satisfies
$R_n(\mu) \subset \supp{\cR_{n}(\mu)}_n \subset B_{d_\lambda}(R_n(\mu))$ .
Observe that this in turn implies that  for all $n\ge n_{d_\lambda}$
\begin{enumerate}
\item [(i)"]  if $p\in \lambda\setminus \mu$, then $\supp{\cV_{n_\lambda,p}}_{n_\lambda}  \cap \supp{\cR_n(\mu)}_n = \varnothing$, and
\item [(ii)"]  if $q\not\in\alpha\in\sO(\sP)$, then $\cl \bigl(\supp{\cR_n(\mu)}_n \setminus \supp{\cU_{n_\lambda}(\lambda)}_{n_\lambda} \bigr) \cap R_n(\alpha)  = \varnothing$.
\end{enumerate}

Recall that for $n\ge n_{\lambda}$, $\cX_n$ is a refinement of $\cX_{n_\lambda}$.  Thus for each $n$ there exists unique
sets $\cV^n_{n_\lambda,p}, \cU^n_{n_\lambda}(\alpha) \subset \cX_{n}$ such that
\begin{equation}
\label{eq:defnVU}
\supp{\cV^n_{n_\lambda,p}}_n = \supp{\cV_{n_\lambda,p}}_{n_\lambda}
\quad\text{and}\quad
\supp{\cU^n_{n_\lambda}(\alpha)}_n = \supp{\cU_{n_\lambda}(\alpha)}_{n_\lambda}. 
\end{equation}
By assumption $\cU_{n_\lambda}(\lambda)\in \IS^-(\cX_{n_\lambda},\cF_{n_\lambda})$.
By Proposition~\ref{prop:cofiltmap} $\cU^n_{n_\lambda}(\lambda)\in \IS^-(\cX_{n},\cF_n)$.

Let $n_{\mu\cup\lambda}\ge n_{d_\lambda}$.  
Then $\cU^{n_{\mu\cup\lambda}}_{n_\lambda}\colon \sO(\lambda^\top)\to \sInvset^{-}(\cX_{n_{\mu\cup\lambda}},\cF_{n_{\mu\cup\lambda}})$ is a partial lift of $R_{n_{\mu\cup\lambda}}$ through $\alpha(\supp{\cdot})\colon \sInvset^{-}(\cX_{n_{\mu\cup\lambda}},\cF_{n_{\mu\cup\lambda}})\to \sR_{n_{\mu\cup\lambda}}$.
To complete the induction step we must show that this partial lift can be extended to $\sO((\lambda\cup\mu)^\top)$.
\vskip 6pt
\noindent {\em Claim:} This partial lift can be extended  to $\sO((\lambda\cup\mu)^\top)$ via the following definition.
\vskip 6pt\noindent
Given $\alpha\in\lambda\cup\mu$ define
$\cU_{n_{\mu\cup\lambda}}\colon \sO((\lambda\cup\mu)^\top) \to  \IS^-(\cX_{n_{\mu\cup\lambda}},\cF_{n_{\mu\cup\lambda}})$ by
\begin{equation}\label{eqn:liftk}
\cU_{n_{\mu\cup\lambda}}(\alpha)=\bigcup_{p\in\alpha}\cV_{n_{\mu\cup\lambda},p}
\end{equation}
where 
\begin{equation}
\label{eqn:defnext}
\cV_{n_{\mu\cup\lambda},q} := \cR_{n_{\mu\cup\lambda}}(\mu)\setminus \cU^{n_{\mu\cup\lambda}}_{n_\lambda}(\lambda)
\quad\text{and}\quad
\cV_{n_{\mu\cup\lambda},p} := \cV^{n_{\mu\cup\lambda}}_{n_\lambda,p}\ \text{for}\ p\in\lambda.
\end{equation}

Our induction hypothesis and the proof of the claim makes use of conditions {\bf C1} - {\bf C3}, thus we begin by verifying that they are satisfied.
By the induction hypothesis to prove {\bf C1} for all $p,p' \subset \lambda \cup \mu$ it is sufficient to show that
$\cV_{n_{\mu\cup\lambda},q} \cap \cV_{n_{\mu\cup\lambda},p} = \varnothing$ for $p\in\lambda$. 
This follows from the fact that
\[
\cV_{n_{\mu\cup\lambda},p} = \cV^{n_{\mu\cup\lambda}}_{n_\lambda,p} \subset \cU^{n_{\mu\cup\lambda}}_{n_\lambda}(\lambda)
\] 
for all $p\in \lambda$.
To prove Condition {\bf C2} observe that by definition 
\[
\supp{\cV_{n_{\mu\cup\lambda},q}} = \supp{\cR_{n_{\mu\cup\lambda}}(\mu)\setminus \cU^{n_{\mu\cup\lambda}}_{n_\lambda}(\lambda)} = \cl \bigl( \supp{\cR_{n_{\mu\cup\lambda}}(\mu)} \setminus |\cU_{n_{\mu\cup\lambda}}(\lambda)|\bigr)
\]
where the latter equality follows from Lemma~\ref{lem:intregdisj}, and then apply (ii)''.
Turning to Condition {\bf C3}, by definition 
$\cV_{n_{\mu\cup\lambda},q} \subset \cR_{n_{\mu\cup\lambda}}(\mu)$
and thus by (i)", \eqref{eq:defnVU}, and \eqref{eqn:defnext} we have that 
$\supp{\cV_{n_{\mu\cup\lambda},p}}\cap \supp{\cV_{n_\lambda,q}} = \varnothing$ for
all $p\in \lambda \setminus \mu$. 
Let $p\in \lambda\cup \mu$. Note that $p\le q$ if and only if $p\in \mu$, and thus $p\not \le q$ if and only if 
$p\in  (\lambda\cup \mu)\setminus \mu$.
Moreover, $q\le p$ if and only if $p=q$, and thus $q\not \le p$ if and only if $p\not = q$.
We conclude that $p~\Vert~q$ if and only if $p\in \lambda \setminus \mu$, which proves {\bf C3} for all $p,p'\in \lambda \cup \mu$ satisfying $p~\Vert~p'$.

To prove the claim we need to verify four statements:
\begin{enumerate}
\item  $\cU_{n_{\mu\cup\lambda}}$ is an extension of $\cU_{n_\lambda}$,
\item  $\cU_{n_{\mu\cup\lambda}}$ maps into $\sInvset^{-}(\cX_{n_{\mu\cup\lambda}},\cF_{n_{\mu\cup\lambda}})$, 
\item $\cU_{n_{\mu\cup\lambda}}$ is a lattice homomorphism, and 
\item $\cU_{n_{\mu\cup\lambda}}$ is a partial lift of $R_{n_{\mu\cup\lambda}}$.
\end{enumerate}

To prove the first statement observe that each $\alpha\subset\lambda\cup\mu$ can be expressed as $\alpha=\beta\cup\nu$ for $\beta\subset\lambda$ and $\nu=\varnothing$ or $\nu=\mu$.  
If $\nu=\varnothing$, then $\cU_{n_{\mu\cup\lambda}}(\alpha)=\cU_{n_\lambda}(\beta)$,
and if $\nu=\mu$, then $\cU_{n_{\mu\cup\lambda}}(\alpha)=\cU_{n_\lambda}(\beta) \cup \cU_{n_{\mu\cup\lambda}}(\mu)$.  The second statement follows from the fact that $\vee = \cup $ as the lattice operation in 
$\sInvset^{-}(\cX_{n_{\mu\cup\lambda}},\cF_{n_{\mu\cup\lambda}})$.  The third statement follows from \eqref{eqn:liftk} and {\bf C1}.
See \cite[Theorem 4.8 {\em Proof of (a)}]{KMV-1a} for details.
To demonstrate the fourth statement, note that
\begin{eqnarray*}
\cU_{n_{\mu\cup\lambda}}(\mu) &=&
\cV_{n_{\mu\cup\lambda},q}\cup \left( \bigcup_{p\in\lambda\cap\mu}\cV_{n_{\mu\cup\lambda},p}\right) \\
&=&
\left( \cR_{n_{\mu\cup\lambda}}(\mu)\setminus \left( \bigcup_{p\in\lambda}\cV_{n_{\mu\cup\lambda},p} \right) \right)
\cup \left( \bigcup_{p\in\lambda\cap\mu}\cV_{n_{\mu\cup\lambda},p} \right)\\
&=&\left( \cR_{n_{\mu\cup\lambda}}(\mu)\setminus \left( \bigcup_{p\in\lambda\cap\mu}\cV_{n_{\mu\cup\lambda},p} \right) \right)\cup\left( \bigcup_{p\in\lambda\cap\mu}\cV_{n_{\mu\cup\lambda},p} \right)\\
&=&\cR_{n_{\mu\cup\lambda}}(\mu)\cup \left( \bigcup_{p\in\lambda\cap\mu}\cV_{n_{\mu\cup\lambda},p} \right) \\
&=& \cR_{n_{\mu\cup\lambda}}(\mu)\cup \cU_{n_{\mu\cup\lambda}}({\lambda\cap\mu}).
\end{eqnarray*}
Therefore, 
\begin{eqnarray*}
\alpha(\supp{\cU_{n_{\mu\cup\lambda}}(\mu)}) &=&
\alpha( \supp{\cR_{n_{\mu\cup\lambda}}(\mu)}) \cup \alpha(\supp{\cU_{n_{\mu\cup\lambda}}({\lambda\cap\mu})})\\
&=&R_{n_{\mu\cup\lambda}}(\mu)\cup R_{n_{\mu\cup\lambda}}({\lambda\cap\mu})=R_{n_{\mu\cup\lambda}}(\mu),
\end{eqnarray*}
and thus $\cU_{n_{\mu\cup\lambda}}$ is a partial lift of $R_{n_{\mu\cup\lambda}}$.
 
We have now proved the claim and completed the induction step.  
Since the lattice of repellers $\sR$ is finite, a finite application of the induction argument gives rise to the commutative diagram \eqref{eq:diagramproof} and hence diagram \eqref{eq:repellerliftdiagram}.

The commutative diagram \eqref{diag:comm-liftfull} guarantees that $\supp{\ell_n(\cdot)}\colon \sR \to \sRNbhdR(X,f)$ or equivalently  that $\supp{\ell_n}$ can be viewed as a lift  of the embedding $\sR \to \sRep(X,f)$ through $\alpha\colon \sRNbhdR(X,f) \to \sRep(X,f)$.
The careful reader will note that Condition {\bf C3}  has not yet been used.  
{\bf C3} implies that $\cU_n$ is a well-separated lift.
The definition of well-separated guarantees that $\supp{\ell_n}(\sR)$ is a sublattice in $\sRNbhd(X,f)$,
which is essential for the final claim of the theorem.  Observe that this implies that $\supp{\ell_n}$ can be viewed as a lift  of the embedding $\sR \to \sRep(X,f)$ through $\alpha\colon \sRNbhd(X,f) \to \sRep(X,f)$.
\end{proof}

Theorem~\ref{thm:conv-latt-1} indicates that given a convergent cofiltration of multivalued maps obtained by refinement that the lattice structure of attractors can be realized as a lift to $\sInvset^+(\cX,\cF)$.  
The following theorem shows that a similar result holds if one works with an arbitrary convergent sequence of multivalued maps.
Note that since $\sInvset^+(\cX,\cF)\subset \sASet(\cX,\cF)$, the conclusion of this theorem is weaker than that of Theorem~\ref{thm:conv-latt-1}.

\begin{theorem}
\label{thm:conv-latt-3}
Let $f\colon X \to X$ be a continuous mapping on a
compact metric space $X$.
Let 
$\cF_n\colon  \cX_n \mvmap \cX_n$ be a convergent sequence of outer approximations.
Then for every finite sublattice $\sA \subset \sAtt(X,f)$ there exists an $n_\sA$ such that for all $n\ge n_\sA$
there exists a lift  $\ell_n\colon  \sA \to \sASet(\cX_n,\cF_n)$ of the inclusion map $\iota\colon \sA \rightarrowtail \sAtt(X,f)$ through $\omega(\supp{\cdot})\colon \sASet(\cX_n,\cF_n) \to \sAtt(X,f)$, i.e.\ the following diagram commutes
\[
\begin{diagram}
\node{~} \node{\sASet(\cX_n,\cF_n)} \arrow{s,r}{\omega(|\cdot|)}\\
\node{\sA}  \arrow{e,l,V}{\imath} \arrow{ne,l,..,V}{\ell_n}  \node{\sAtt(X,f).} 
\end{diagram}
\]
Furthermore, $\ell_n$ can be chosen such that $\supp{\ell_n(\sA)}$ is a sublattice of $\sANbhd(X,f)$.
Similar statements hold for finite sublattices  $\sR \subset \sRSet(X,f)$.
\end{theorem}

The proof for Theorem~\ref{thm:conv-latt-3} is similar in spirit to that of Theorem~\ref{thm:conv-latt-2}.  However, because we are not assuming a cofiltration of grids we cannot make use of Proposition~\ref{prop:cofiltmap}.
We make use of the  following abstract result to circumvent this difficulty.

\begin{theorem}\label{thm:lift}
Let $f\colon X \to X$ be a continuous mapping on a compact metric space $X$.
Let $\sP$ be a poset with $I$ elements.
Let  $R\colon\sO(\sP)\to\sR\subset\sRep(X,f)$ be a lattice isomorphism, and let $\pi:\sP\to\{1,2,\ldots,I\}$ be a 
bijective, order-preserving map. 
Let $p_i :=\pi^{-1}(i)$ and $\mu_i :=\down p_i \in \sO(\sP)$ for $i=\setof{1,\ldots,I }$. 
Then there exist $\setof{\epsilon_i}_{i=1}^I$
with $\epsilon_i>0$, such that if $\setof{N_i}_{i=1}^I$
is a collection of compact sets satisfying 
\begin{equation}\label{eqn:stable}
B_{\epsilon_i/2}(R(\mu_i)) \subset N_{i} \subset B_{\epsilon_i}(R(\mu_i)),
\end{equation}
then each $N_i$ is a repelling neighborhood.
Furthermore, $U\colon\sO(\sP)\to\sRNbhd(X,f)$, determined by 
 $U(\mu_1)=N_1$ and $U(\mu_{i+1})=N_{i+1}\cup U(\pred{\mu_{i+1}})$, is a lift
 of $R$ through $\alpha \colon \sRNbhd(X,f) \to \sRep(X,f)$.
\end{theorem}

\begin{proof}
We use an inductive argument to prove the existence of $\{\epsilon_i\}_{i=1}^I$. Simultaneously we prove that
at each stage of the induction argument the restriction of $U$  to $\sO(\mu_i^\top)$, which we denote by $U_i$, is a partial lift of $R$ through $\alpha \colon \sRNbhd(X,f) \to \sRep(X,f)$.  
More precisely, once $\epsilon_i$ is determined we choose a compact set $N_i\subset X$ satisfying \eqref{eqn:stable} at which point $U_i$ is well defined.

Given $U_i$ define
\[
V_{i,p} := U_i(\beta)\setminus U_i(\gamma)\quad \text{for }\beta\setminus\gamma = \setof{p}
\]
This definition is independent of the particular choice of $\beta$ and $\gamma$. See  \eqref{eqn:inatom}
and the associated discussion.
Observe that  
\[
 U_i(\alpha) = \bigcup_{p\in \alpha} V_{i,p}.
\]

Choose $\epsilon_1 >0$ such that $B_{\epsilon_1}(R(\mu_1)) \cap R(\mu_1)^* = \varnothing$.
Choose $N_1$ satisfying \eqref{eqn:stable}.
By \cite[Proposition 3.25]{KMV-1a} $N_1$ is a repelling neighborhood of $R(\mu_1)$. 
This defines $U_1$ on $\sO(\mu_1^\top)=\sO(\setof{p_1}^\top)$.
We leave it to the reader to check that  the following three conditions 
(cf.\ the proof of Theorem~\ref{thm:conv-latt-2}) are trivially satisfied:
\begin{description}
\item[C1] $V_{k,p}\cap V_{k,p'} = \varnothing$ for all $p\not = {p'}$;
\item[C2] $\cl(V_{k,p}) \cap R(\alpha) = \varnothing$ if $p \not\in\alpha$;
\item[C3] $\cl(V_{k,p}) \cap \cl(V_{k,p'}) = \varnothing$ for all $p~\Vert~{p'}$,
\end{description}
where $k=1$, $p,p'\in \setof{p_1},$ and $\alpha\in\sO(\sP)$.  Since 
\[
\alpha(N_1,f) = \alpha(U_1(\setof{p_1}),f) = R(\setof{p_1}),
\]
$U_1$ is a partial lift of $R$ through $\alpha(\cdot,f)$.

To carry out the induction argument, assume that $\{\epsilon_i\}_{i=1}^k$ and $\setof{N_i}_{i=1}^k$ have been chosen
such that \eqref{eqn:stable} is satisfied and that the resulting lattice homomorphism $U_k$ defined on
$\sO(\lambda_k^\top)$, where $\lambda_k := \setof{p_1, \cdots p_k}$, is a partial lift of $R$ through $\alpha(\cdot,f)$
satisfying conditions {\bf C1} - {\bf C3} for $p,p'\in\lambda_k$ and $\alpha\in\sO(\sP).$

Choose $\epsilon^0_{k+1}$ such that $B_{\epsilon^0_{k+1}}(R(\mu_{k+1}))\cap R(\mu_{k+1})^* =\varnothing$.
By \cite[Proposition 3.25]{KMV-1a} if $R({\mu_{k+1}})\subset \Int(D) \subset B_{\epsilon^0_{k+1}}(R(\mu_{k+1}))$, 
then $D$ is a repelling neighborhood for $R({\mu_{k+1}})$.

We claim that there exists  $\epsilon_{k+1}\in (0, \epsilon^0_{k+1})$ such that if 
$R({\mu_{k+1}})\subset \Int(D) \subset B_{\epsilon_{k+1}}(R(\mu_{k+1}))$, then $D$ satisfies 
the following two conditions:
\begin{enumerate}
\item [(i)] if $p\in \lambda_k\setminus {\mu_{k+1}}$, then $\cl(V_{k,p}) \cap D = \varnothing$, and
\item [(ii)] if %$p_{k+1}\not\in\alpha$,||
$p_{k+1}\not\in\alpha\in\sO(\sP)$, then $\cl \left(D\setminus U_k(\lambda_k)\right) \cap R(\alpha) = \varnothing$.
\end{enumerate}
To establish (i) we use the induction hypothesis {\bf C2}, which implies that $\cl (V_{k,p}) \cap R(\mu_{k+1}) = \varnothing$ for all $p\in {\lambda_k}\setminus {\mu_{k+1}}$.  
Since $\cl (V_{k,p})$ and $R(\mu_{k+1})$ are compact, we can choose $\epsilon^1_{k+1}\in (0,\epsilon^0_{k+1})$ such that $\cl(V_{k,p}) \cap D = \varnothing$ for all neighborhoods $D$  such that $R({\mu_{k+1}})\subset \Int(D) \subset B_{\epsilon^1_{k+1}}(R(\mu_{k+1}))$.

To establish (ii) note  that the inclusion
\begin{equation}
\label{eqn:alt2}
D \cap R(\alpha) \subset \Int U_k(\lambda_k)
\end{equation}
 is equivalent to 
\begin{equation}
\label{eqn:alternative}
(D\setminus \Int U_k(\lambda_k)) \cap R(\alpha) = \varnothing.
\end{equation}
Observe that (\ref{eqn:alternative}) implies (ii) and thus it is sufficient to verify  (\ref{eqn:alt2}) under the assumption that  $p_{k+1}\not\in\alpha$.
Since, $R(\mu_{k+1})\cap R(\alpha) = R({\mu_{k+1}\cap \alpha}) \subset R(\lambda_k)\subset \Int U_k(\lambda_k)$, we can choose $\epsilon_{k+1}\in (0, \epsilon^1_{k+1})$ such that 
$R({\mu_{k+1}})\subset \Int(D) \subset B_{\epsilon_{k+1}}(R(\mu_{k+1}))$ satisfies (\ref{eqn:alt2}).

For the above choice of $\epsilon_{k+1}$ choose $N_{k+1}$ satisfying \eqref{eqn:stable}.  This defines $U_{k+1}$.
The proof that $U_{k+1}$ is a partial lift is identical in form to that of the proof of Theorem~\ref{thm:conv-latt-2} and thus
left to the reader.

The induction argument terminates after $I$ steps.
\end{proof}

\begin{proof}[Proof of Theorem~\ref{thm:conv-latt-3}]
We first prove the result in the context of repellers, i.e.\ we prove the existence of a lift $\ell_n\colon  \sR \to \sRSet(\cX_n,\cF_n)$ of the inclusion map $\iota\colon \sR \rightarrowtail \sRep(X,f)$ through $\alpha(\supp{\cdot})\colon \sRSet(\cX_n,\cF_n) \to \sRep(X,f)$.

%\corrl This was not correct and somehow some stuff got deleted --BK 9/5/14 
%<<
%Let $\setof{\epsilon_i}_{i=1}^I$,  be the set of radii produced by applying Theorem~\ref{thm:lift} with $\sP = \sJ(\sR)$. 
%Choose $n$ sufficiently large so that  $\diam(\cX_{n})<\min_{i}\{\epsilon_i/2\}$.
%For $i = 1,\ldots, I$ define $N_i=\cov_{\cX_n}(B_{\epsilon_i/2}(R(\mu_i)))$.  
%By Theorem~\ref{thm:lift}  $U\colon \sO(\sP)\to \sRNbhd(X,f)$ is the desired lift. 
%||
Let $\setof{\epsilon_i}_{i=1}^I$,  be the set of radii produced by applying Theorem~\ref{thm:lift} with $\sP = \sJ(\sR)$. 
Choose $n$ sufficiently large so that  $\diam(\cX_{n})<\min_{i}\{\epsilon_i/2\}$.
For $i = 1,\ldots, I$ define $\cN_i=\cov_{\cX_n}(B_{\epsilon_i/2}(R(\mu_i)))$.  
By Proposition~\ref{prop:conv-latt-3}, if $n$ is chosen sufficiently large, then each $\cN_i\in\sRSet(\cX_n,\cF_n)$.
Similarly to Theorem~\ref{thm:lift}, the map $\ell_n\colon\sO(\sP)\to\sRSet(\cX_n,\cF_n)$ given by
$$
\ell_n(\mu_1)=\cN_1\;\;\hbox{and}\;\; \ell_n(\mu_{i+1})=\cN_{i+1}\cup \ell_n(\pred{\mu_{i+1}})
$$
is a lift.
By Theorem~\ref{thm:lift}  $|\ell_n|\colon \sO(\sP)\to \sRNbhd(X,f)$  is a lift. 
%>>

The statement for attractors follows from duality, i.e.\ by the proof of \cite[Theorem 1.2]{KMV-1a} and in particular
\cite[commutative diagram (24)]{KMV-1a}.
\end{proof}

\begin{remark}\label{rmk:mono}
%\corrl not sure what to say about this at this point -- perhaps we should just get rid of it? -- BK 9/6/14 , --- I would leave it the way it is, RC 9/17/14 -- I think we agreed to remove this since it is a definition in the last 
%sentence of the paper that really doesn't have much content at this point. Instead we put a
%comment about monotonicity and contrast the 2 theorems. This remark is referenced earlier 
%in the paragraph that follows the definition of cofiltration of mappings. -- BK 9/17/14  agree with changes KM 9/18/14
%<<
%A lift which can be constructed to satisfy Property~$(\ref{eqn:stable})$ will be called a {\em stable lift}.
%%
%If we weaken the lifts to the lattices of attracting sets and repelling sets instead of 
%forward invariant sets  and backward invariant sets, then convergence can be obtained for arbitrary convergent sequences of multivalued mappings. Note that it is not clear whether
%such lifts   necessarily factor through the lattices of forward invariant sets and backward invariant sets
%respectively. 
%||
To put Theorems~\ref{thm:conv-latt-2} and~\ref{thm:conv-latt-3} into perspective, we recall that 
the monotonicity of images of $\cF_n$ required for a cofiltration of mappings may in some applications
be computationally expensive to attain, even though most practical algorithms construct
$\cF_n$ on a cofiltration of grids $\cX_n$ through successive refinement.
Theorem~\ref{thm:conv-latt-2} implies that if one does indeed compute a cofiltration of mappings,
then the structure of attractors of $f$ can be realized in forward invariant sets of $\cF$.
Without a cofiltration, Theorem~\ref{thm:conv-latt-3} still implies the weaker result that the
 structure of attractors of $f$ can be realized in attracting sets of $\cF$.
%>>
\end{remark}

\begin{remark}
\label{rem:comm-lift11}
By Proposition \ref{prop:discr-att-to-varphi3} and \ref{prop:discr-att-to-varphi4} we can restate the Diagram (\ref{diag:comm-liftfull}) by Diagram
\eqref{diag:AR2a} for $\varphi$. For
 attractors this reads:

\begin{equation*}
%\label{diag:comm-lift12}
\begin{diagram}
\node{\IS^+(\cX,\cF)}\arrow{se,l,A}{\bomega}\arrow{e,l,V}{\imath}\node{\sASet(\cX,\cF)}\arrow{e,l,V}{|\cdot|}\arrow{s,l,A}{\bomega}\node{\sANbhdR(X,\varphi)}\arrow{e,l,A}{\omega}\node{\sAtt(X,\varphi)}\\
\node{}\node{\sAtt(\cX,\cF)}\arrow{ene,r}{\omega(|\cdot|)}
\end{diagram}
\end{equation*}
By Corollary \ref{for:sameA} we have that $\sAtt(X,f) = \sAtt(X,\varphi)$ as lattices, and therefore
$\ell$ also provides a lift for $\varphi$ in Theorem~\ref{thm:conv-latt-1} and Theorem~\ref{thm:conv-latt-2}.
In the case of Theorem~\ref{thm:conv-latt-1}, when we construct lifts under refinements, then
the lift $\ell$ yields opposite arrows for all arrows in Diagram (\ref{diag:comm-lift11}).
In the case of Theorem~\ref{thm:conv-latt-2} we only provide opposite arrows to $\sASet(\cX,\cF)$, which implies that
lifts to $\IS^+(\cX,\cF)$ may not exist.  The same reasoning holds for repellers and their lifts.
\end{remark}
% the realization of attracting sets yield regular attracting neighborhoods and the mapping $|k|$ provides a
%lift to $\sANbhdR(X,\varphi)$.

\bibliographystyle{plain}
\bibliography{KMV1b-biblist}

% \appendix
%\input{attractors-appendix-IIs}
%\vskip1cm

\end{sloppypar}
\end{document}